\RequirePackage{fix-cm}
\documentclass[smallextended]{svjour3}       % onecolumn (second format)
\smartqed  % flush right qed marks, e.g. at end of proof
\usepackage{graphicx}
\usepackage{endnotes}
\usepackage{mathtools} % for coloneqq
\usepackage{enumitem} % for enumeration
\usepackage{dsfont} % for indicator function
\usepackage{mathrsfs}  
\usepackage{verbatim}
\usepackage{romannum}
\usepackage{graphicx}
\usepackage{float}
\usepackage{caption}
\usepackage{subcaption}
\usepackage{braket}
\usepackage{scalerel}
\usepackage{bm}
\usepackage{xcolor}
\usepackage[]{natbib}
\usepackage[breaklinks]{hyperref}
\usepackage{microtype}
\usepackage[ruled]{algorithm2e}
\usepackage{array}
\usepackage{amsfonts}

\DeclareMathOperator*{\argmin}{arg\,min}
\newcommand{\ba}{{\bf a}}
\newcommand{\bb}{{\bf b}}

\newcommand{\bd}{{\bf d}}

\newcommand{\bh}{{\bf h}}

\newcommand{\bu}{{\bf u}}

\newcommand{\bw}{{\bf w}}
\newcommand{\bx}{{\bf x}}

\newcommand{\bz}{{\bf z}}
\newcommand{\bba}{{\bf A}}

\newcommand{\bbd}{{\bf D}}

\newcommand{\bbk}{{\bf K}}

\newcommand{\bbw}{{\mathcal{W}}}
\newcommand{\bbz}{{\bf Z}}

\newcommand{\Exp}{\mathbb{E}}

\newcolumntype{P}[1]{>{\centering\arraybackslash}m{#1}}

\DeclareMathOperator{\E}{\mathbb{E}}

\journalname{Mathematical Programming}
\begin{document}
\renewcommand{\thepage}{\arabic{page}}

\title{Multistage Stochastic Optimization via Kernels%\thanks{Grants or other notes
}


\author{Dimitris Bertsimas         \and
        Kimberly Villalobos Carballo %etc.
}


\institute{D. Bertsimas \at
              Boeing Professor of Operations Research \at Massachusetts Institute of Technology, Cambridge, MA 02139 \\
              \email{dbertsim@mit.edu}           %  \\
           \and
           K. Villalobos Carballo \at
              Operations Research Center, Massachusetts Institute of Technology, Cambridge, MA 02139\\
              \email{kimvc@mit.edu}
}

\date{Received: date / Accepted: date}

\maketitle

\begin{abstract}
We develop a non-parametric, data-driven, tractable approach for solving multistage stochastic optimization problems in which decisions do not affect the uncertainty. The proposed framework represents the decision variables as elements of a reproducing kernel Hilbert space and performs functional stochastic gradient descent to minimize the empirical regularized loss. By incorporating sparsification techniques based on function subspace projections we are able to overcome the computational complexity that standard kernel methods introduce as the data size increases. We prove that the proposed approach is asymptotically optimal for multistage stochastic optimization  with side information. Across various computational experiments on stochastic inventory management problems, {our method performs well in multidimensional settings} and remains tractable when the data size is large. Lastly, by computing lower bounds for the optimal loss of the inventory control problem, we show that the proposed method produces decision rules with near-optimal average performance.
%\subclass{MSC code1 \and MSC code2 \and more}
\end{abstract}
\section{Introduction}
Multistage stochastic optimization arises in numerous applications (e.g., supply chain management, energy planning, inventory management among others) and remains an important research area in the optimization community \citep{Birge2011, shapiro2014lectures, bertsimas2011theory}. In these problems, the decision variables are split across multiple periods and decisions are made sequentially as more information becomes available. The goal is to make high quality decisions that minimize the expectation of a given cost function by accurately modeling future uncertainty.  In practice, decision makers can use historical data to get a sense of the future uncertainty. For instance, consider a retailer selling products with short life cycles who needs to make frequent orders to restock inventory without knowing the future demands. To minimize costs the retailer must use the remaining inventory quantities as well as historical data to gain insight into future demands. Another example is energy planning, in which operators decide daily production levels without knowing how weather conditions will affect the output of the wind turbines. In this case historical wind patterns are valuable for better planning.

Besides historical data, auxiliary covariates are often available and can help predict uncertainty. For example, in the fashion industry, color and brand are useful factors to predict demand of a new item. Accordingly, recent work has focused on using predictive analytics to leverage available side information and historical data to make better decisions. \cite{ban2019dynamic} for instance, fit covariate and historical data to a regression model and prove theoretical guarantees for the dynamic procurement problem. Another approach is that of \cite{bertsimas2018data}, which considers an uncertainty set around each data sample and applies robust optimization tools to find linear decision rules that are asymptotically optimal under mild conditions. This framework was generalized in
\cite{bertsimas2019predictions}, where machine learning methods are incorporated to find weights that produce more accurate approximations of the objective. However, these dynamic methods are affected by the curse of dimensionality; they require scenario tree enumeration and can require many hours for solving problems with only a few stages.

In this paper, we propose a  non-parametric, data-driven and tractable approach to solving multistage stochastic optimization problems. By restricting the decision variables to be in a reproducing kernel Hilbert space (RKHS) generated by  a universal kernel, we can approximate a large class of functions using non-parametric functional representations. We incorporate sparsification techniques based on function subspace projections that allow our proposed algorithm  to overcome the complexity growth that kernel methods introduce when directly applying the Representer Theorem to large data sets. The input to our algorithm is historical data and we make no assumptions on the correlation structure of the uncertainties across stages. We perform computational experiments on real-world   multistage stochastic problems, and show how our method not only produces near optimal solutions but also remains tractable in higher dimensions and with large data sizes. 
\subsection{Related Literature}
Kernel methods have been used in recent work to solve stochastic multistage optimization problems with side information. \cite{Hanasusanto2013}, for example, approximate the objective using kernel regression, and \cite{pflug2016empirical} apply a kernel density estimator to the historical data to develop a non-parametric predict-then-optimize approach that comes with asymptotic optimality guarantees under strong conditions. However, these are local methods in which the predictions are made based only on those data points that are similar to the current observation. As noted in \cite{Bertsimas2020b}, such approaches require more data and perform worse on high dimensions compared to global methods, which instead optimize over functional variables that make the predictions. 

The Machine Learning community has long applied kernel methods to solve online learning problems \citep{wheeden2015measure, norkin2009stochastic}, but they have focused purely on predictive and not on prescriptive tasks. More recently, \cite{Bertsimas2020b} has aimed to extend kernel methods to data-driven, single-period optimization problems with auxiliary information by using the Representer Theorem to transform the optimization over functions into an optimization over parameters. They show that this approach overcomes the curse of dimensionality; however, its main disadvantage is that the number of parameters per decision grows linearly with the number of observations, resulting in function representations that are as complex as the size of the data and that become potentially intractable especially in multistage settings.

Works on stochastic optimization in a RKHS have developed multiple heuristics to reduce the number of parameters in the function representation. For instance, \cite{zhang2013online} uses random dropping, \cite{kivinen2004online} introduces forgetting factors, and \cite{honeine2011online} as well as \cite{engel2004kernel}  apply compressive sensing techniques. These approaches sucessfully achieve sparser functional representations, but they usually produce suboptimal approximations \citep{honeine2011online, engel2004kernel}. 

We instead follow the approach from \cite{koppel2016parsimonious} of applying Functional Stochastic Gradient Descent (FSGD) and projecting the iterates onto sparse subspaces that are found by removing parameters associated with data points that do not contribute much to the value of the decisions \citep{pati1993orthogonal}. This approach maintains optimality while addressing the  complexity growth that kernel methods exhibit as the data size increases. Intuitively, since stochastic gradient descent iterates are a noisy signal for the optimal solution, by projecting the iterates to have small model order we can ignore some of the noise while preserving the goal signal. %DB in what sense? %KIM instead of saying that projecting is more important, I am instead explaining  intuitively why it works.
The sparse subspaces of the RKHS onto which projections are made can be effectively found using kernel orthogonal matching pursuit \citep{vincent2002kernel}, an algorithm which given a function $f$ and an error bound $\epsilon$, generates a sparse
approximation of $f$ that is in a neighborhood of $f$ of radius $\epsilon$ in Hilbert norm. \cite{koppel2016parsimonious} show that for a specific choice of $\epsilon$ and of step-size for the FSGD algorithm, the projected FSGD iterates produce decisions that converge in mean to the optimal solution. 
\subsection{Contributions}
In this paper, we propose a novel data-driven approach for solving multistage stochastic optimization problems with side information using kernels. Specifically, we represent the controls as elements of a reproducing kernel Hilbert space and use loss-minimizing machine learning methods to predict them. In addition, we incorporate sparsification techniques to reduce the total number of parameters per control. We prove that this approach is asymptotically optimal, guaranteeing near optimal approximations for problems with large amounts of data. We also show that our approach remains computationally tractable in high dimensions and with large data sizes. In detail, our contributions are as follows.
\begin{enumerate}
    \item We propose a novel data-driven approach for multistage stochastic optimization problems with side information based on reproducing kernel Hilbert spaces. The approach takes as input historical data and minimizes the regularized empirical loss by applying functional stochastic gradient descent to optimize the decision rules, i.e.,  the functions which specify what decision to make in each stage. To the best of our knowledge, this is the first tractable application of reproducing kernel Hilbert spaces to multistage optimization problems with large data sizes. {While a kernel based formulation of the multistage stochastic optimization problem is briefly suggested (without any computational experiments) in \cite{Bertsimas2020b}, their non-stochastic and non-sparse approach is not tractable for large data sizes since both time and memory requirements increase cubically with the amount of data.}\\
    \item We extend sparsification techniques used by \cite{koppel2016parsimonious} to multistage optimization settings in order to reduce both, space and time complexities of our algorithm. Specifically, we use Functional Stochastic Gradient Descent (FSGD) to minimize the objective and project each iterate onto a sparse subspace that is found by removing parameters corresponding to data points with small contributions. We show that applying FSGD without any sparsification results in methods that do not scale to larger number of periods or data sizes. If sparsity is not added, the computational cost and the storage requirement increase quadratically with the data size. With the proposed method, however, both space and time complexities present linear growth with a constant factor that depends on the step size of the FSGD algorithm.\\
    \item We prove that if the loss function is convex, Lipschitz and differentiable almost everywhere, then the expected loss achieved with our algorithm converges in probability to the expected loss of the optimal decision rules in the space of continuous functions. \\ % In addition, if the problem is subject to constraints that are also convex, Lipschitz and differentiable almost everywhere, then we also prove asymptotic optimality by embedding the constraints into the objective with a penalty function and showing that the expected constraint violation vanishes as the penalty parameter increases.\\
  %DB What is n ? Do you mean N?  This tearms makes little sense as n is not defined but the nth step. Please correct here and in the corresponsing theorem. 
    %KIM fixed!
    \item We demonstrate across several instances of   inventory management problems that the proposed method finds near-optimal solutions using only a few parameters and with very low computational times. We show that increasing the number of periods, the dimension of the data, the dimension of the controls or the data size does not affect the tractability of our approach. 
\end{enumerate}
The paper is organized as follows: Section \ref{sec:prob_setting} introduces the exact framework for the problem being solved, Section \ref{sec:opt} contains the data-driven formulation of the multistage stochastic optimization problem with side information, Section \ref{sec:algorithm} presents the proposed algorithm, Section \ref{sec:convergence} states the convergence theorems, Section \ref{sec:complexity} analyses the complexity of the proposed method, and Section \ref{sec:experiments} shows the results for several computational experiments.

\section{Problem Setting} \label{sec:prob_setting}
We consider a discrete-time, convex, multistage stochastic problem over a finite horizon $T$. Initially, we observe some auxiliary covariates $\bx \in \mathcal{X}\subseteq \mathbb{R}^{q_0}$. Then, random disturbances $\bw_t$ that belong to a known set $\bbw_t\subseteq \mathbb{R}^{q_t}$ are sequentially observed over time. At every stage $t$, after observing the covariates $\bx$ and the previous disturbances $ (\bw_1,\ldots, \bw_{t-1})$, a decision $\bu_t\in \mathbb{R}^{r_t}$ is made. The total cost for the observed sequence of covariates, disturbances and decisions is $c(\bu_1,\ldots, \bu_T, \bx, \bw_1,\ldots, \bw_{T})$. 

A standard decision rule $\bar{\bu}(\cdot) = (\bar{\bu}_1(\cdot),\ldots, \bar{\bu}_T(\cdot))$ consists of functions $\bar{\bu}_t :  \bbw_1\times\ldots\times \bbw_{t-1} \rightarrow{\mathbb{R}^{r_t}}$ that at each stage $t$ take as input the disturbances up to that point and output a decision for the given stage. Specifically, denoting $\bw \coloneqq (\bw_1,\ldots, \bw_T)$ and $\bw_{1:t} \coloneqq (\bw_1,\ldots,\bw_{t})$, we have that the standard decision rule $\bar{\bu}(\cdot)$ applied to $\bw$ outputs $
\bar{\bu}(\bw) = \big(\bar{\bu}_1,\bar{\bu}_2(\bw_{1:1}),\ldots, \bar{\bu}_T(\bw_{1:T-1})\big)$. The multistage optimization problem over the space of continuous decision rules $\hat{\mathcal{F}}$ conditioned on some observed covariates $\bx_0$ can then be written as
\begin{equation}\label{main}
\begin{array}{rl}
\displaystyle\min_{\bar{\bu} \in \hat{\mathcal{F}}} \quad \:  & \displaystyle \Exp_{\bw | \bx} \left[ c(\bar{\bu}(\bw),\bx, \bw)\; |\; \bx = \bx_0 \right],
\end{array}
\end{equation}
where $c(\cdot)$ is a convex loss function. As noted in \cite{Bertsimas2020b}, the conditional problem in Eq. \eqref{main} can be formulated as an unconditional optimization problem by augmenting the domain of the decision rules to also take the covariates as input, and then evaluating the observed covariates in the decision rules found. In this paper, we adopt the same approach and therefore we consider augmented decision rules $\bu(\cdot) = \big(\bu_1(\cdot),\ldots, \bu_T(\cdot)\big)$ with augmented domains $\bu_t : \mathcal{X}\times \bbw_1\times\ldots\times \bbw_{t-1} \rightarrow{\mathbb{R}^{r_t}}$. The augmented decision rule applied to the data point $\bw$ with covariates $\bx$ outputs
\[
\bu(\bx,\bw) = \big(\bu_1(\bx),\bu_2(\bx,\bw_{1:1}),\ldots , \bu_T(\bx, \bw_{1:T-1})\big).
\]
From now on we will join the covariates and the disturbances into a single random variable $\bz\coloneqq (\bx, \bw)$ to simplify notation, and we index $\bz$ starting at time $0$ instead of time $1$, so that $\bz_{0:t}\coloneqq(\bx, \bw_1,\ldots,\bw_{t})$. Defining $\mathcal{F}$ as the space of continuous augmented decision rules, and $\mathcal{Z}\coloneqq \mathcal{X}\times\bbw_1\times\ldots\times \bbw_T$, we obtain that solving Eq. \eqref{main} is equivalent to solving the problem
\begin{equation}\label{main_unconditional}
\begin{array}{rl}
\displaystyle\min_{\bu \in \mathcal{F}} \quad \:  & \displaystyle \Exp_{\bz} \big[ c\big(\bu(\bz), \bz\big) \big]  
\end{array}
\end{equation}
and evaluating the optimal solution $\bu^*(\cdot)$ at $\bx = \bx_0$ to obtain the standard decision rule $\bar{\bu}^*(\bw) = \bu^*(\bx_0, \bw)$.

\section{Reproducing Kernel Hilbert space formulation for Multistage Optimization} \label{sec:opt}
%\subsection{Method}
We now propose a data-driven approach for multistage stochastic optimization problems with side information based on a Reproducing Kernel Hilbert space (RKHS). We include an onverview of these spaces in Appendix \ref{sec:rkhs}. We will assume that we have historical observations  $\mathcal{S} \!=\! \{\bz^n\}_{n = 1}^N\! = \{(\bx^n\!,\! \bw_1^n,\ldots, \bw_T^n)\}_{n = 1}^N$ that are independently and identically distributed according to some unknown distribution. Let $K_t :\mathcal{X}\times \bbw_1\times\ldots\times \bbw_{t-1} \rightarrow \mathbb{R}$ be a positive universal kernel and $\mathcal{H}_t$ the reproducing Kernel Hilbert space generated by $K_t$. We consider the Cartesian product Hilbert space, $
\mathcal{H} \coloneqq \mathcal{H}_1^{r_1}\times \ldots\times \mathcal{H}_T^{r_t}$
with inner product defined by
\begin{align*}
&\Big\langle \!\!\big((u_{1,1},\mathinner{{\ldotp}{\ldotp}{\ldotp}}, u_{1,{r_1}}),\mathinner{{\ldotp}{\ldotp}{\ldotp}}, (u_{T,1},\mathinner{{\ldotp}{\ldotp}{\ldotp}}, u_{T, {r_T}})\big),\big((v_{1,1},\mathinner{{\ldotp}{\ldotp}{\ldotp}}, v_{1,{r_1}}),\mathinner{{\ldotp}{\ldotp}{\ldotp}}, (v_{T,1},\mathinner{{\ldotp}{\ldotp}{\ldotp}}, {v}_{T, {r_T}})\big)\!\! \Big\rangle_{\mathcal{H}} \\
&\coloneqq \sum_{t = 1}^T\sum_{i = 1}^{r_t} \langle u_{t, i}, v_{t, i}\rangle_{\mathcal{H}_t},
\end{align*}
where $\langle u, v\rangle_{\mathcal{H}_t}$ corresponds to the inner-product between $u$ and $v$ with respect to the Hilbert space $\mathcal{H}_t$. We can approximate the solution of problem \eqref{main_unconditional} by applying its empirical regularized version and restricting the decision rules to be in $\mathcal{H}$:
\begin{equation}\label{directemp}
\min_{\bu \in \mathcal{H}}\: \frac{1}{N} \sum_{n=1}^N c(\bu(\bz^n), \bz^n)  + \frac{\lambda}{2} \|\bu\|_{\mathcal{H}}^2.
\end{equation}
Even though problem \eqref{directemp} is not equivalent to problem \eqref{main_unconditional}, if $\lambda$ vanishes with the data size then the regularized empirical average becomes a closer estimate of the expectation as $N$ increases. We will then focus on solving problem \eqref{directemp}, and later
in Corollary \ref{corollary}  
we show that as the data size goes to infinity, the expected loss converges in probability to the optimal solution of problem \eqref{main_unconditional}. 

One way to solve the regularized empirical problem \eqref{directemp} is to use the multidimensional version of the Representer Theorem \citep{wahba1990spline, soentpiet1999advances, scholkopf2002learning,  shawe2004kernel}, which says that for each $t=1,\ldots, T$ there exists a scalar matrix $\bba_t$ such that the optimal solution to \eqref{directemp} satisfies 
\[\bu_t(\cdot) = \bba_t \boldsymbol{K}_t(\bbz_t, \cdot) ,\]
where $\boldsymbol{K}_t(\bbz, \cdot) \coloneqq [K_t(\bz^1, \cdot),\ldots, K_t(\bz^N, \cdot)]^T$, and the time subscript for a data matrix $\bbd = [\bd^1, \ldots, \bd^N]$ refers to $\bbd_t = [\bd^1_{0:{t-1}}, \ldots, \bd^N_{0:{t-1}}]$. However, with this approach each decision $u_{t, i}$ has as many scalar parameters as data points, which generates both memory and performance problems as the number of data points becomes large. We instead want an algorithm for which more data yields better results overall, without increasing its complexity or worsening performance. General sparisification techniques like those found in \cite{kivinen2004online}, \cite{zhang2013online} or \cite{engel2004kernel}, successfully reduce the number of parameters; however they do so at the cost of compromising  optimality. We therefore take the pruning approach developed in \cite{koppel2016parsimonious} to solve problem $\eqref{directemp}$; we apply functional gradient descent to minimize the objective and at each iteration we drop those parameters that add near zero contribution to the value of the decisions, ensuring convergence to an optimal solution.
\section{Sparse Multistage Optimization with Kernels}\label{sec:algorithm}
In this section, we extend sparsification techniques used by \cite{koppel2016parsimonious} to the multistage optimization setting described in the previous section in order to reduce both, space and time complexities of our algorithm. Specifically, we describe an iterative algorithm for solving \eqref{directemp} using Functional Stochastic Gradient Descent and sparse projections. In order to ease notation, we first make the following definitions for an augmented decision rule $\bu$:
\begin{align}
E(\bu) & \coloneqq \mathbb{E}_{\bz}\left[c(\bu(\bz), \bz)\right], \label{expected_loss}\\
E^{\lambda}(\bu)& \coloneqq E(\bu)  + \frac{\lambda}{2} \|\bu\|_{\mathcal{H}}^2, \label{exp_reg_loss}\\
E^{\lambda}_\mathcal{S}(\bu)&  \coloneqq \frac{1}{N}\sum_{n=1}^N c\big(\bu(\bz^n), \bz^n\big)  + \frac{\lambda}{2} \|\bu\|_{\mathcal{H}}^2, \label{emp_reg_loss}\\
E^{\lambda}_n(\bu)&  \coloneqq  c\big(\bu(\bz^n),\bz^n\big)  + \frac{\lambda}{2} \|\bu\|_{\mathcal{H}}^2 \label{stoc_reg_loss}.
\end{align}
The algorithm relies on the fact that the expectation of $E^{\lambda}_n(\bu)$ over data yields $E^{\lambda}(\bu)$ to make stochastic gradient updates that converge to the optimal solution, while at the same time removing unnecessary parameters along the descent trajectory. 
\subsection{Functional Stochastic Gradient Descent {(FSGD)}}
Thanks to the fact that a RKHS preserves distance and to the continuity properties of real spaces, a derivative with respect to an element $f$ of a RKHS (a function) can be well defined and it satisfies the standard properties of derivatives of real functions. Following \cite{kivinen2004online}, we can then derive a generalization of the Stochastic Gradient Descent algorithm for elements of $\mathcal{H}$. This method is referenced as \textit{functional stochastic gradient descent}.

We compute the gradient of  $E^{\lambda}_n(\bu)$ with respect to the functions $\bu$ using the identity $u_{t,i}(\bz_{0:t-1}) = \langle K(\bz_{0:t-1}, \cdot), u_{t,i} \rangle_{\mathcal{H}}$, which is known as the \textit{reproducing property} of kernels. Differentiating on both sides of this equation we obtain
\begin{align}
\label{diff_kernel}
    \frac{\partial u_{t,i}(\bz_{0:t-1})}{\partial u_{t,i} } = \frac{\partial \big\langle u_{t,i}, K_t(\bz_{0:t-1}, \cdot ) \big\rangle}{\partial u_{t,i}} =  K_t(\bz_{0:t-1}, \cdot), \:\: \forall \: \: i\in [r_t], \: \: t\in [T],
\end{align}
where $[K]=\{1,\ldots, K\}$. The stochastic functional gradient can then be computed using the chain rule:
\begin{align}
\nabla_{\bu_t}c\big(\bu(\bz^n), \bz^n\big) &= \nabla_{_{\scaleto{\bu_t(\bz_{\scaleto{0:t-1}{4pt}})}{6pt}}} c(\bu(\bz^n), \bz^n)\:  K_t(\bz^n_{0:t-1}, \cdot),\label{grad_computation}\\
\implies \nabla_{\bu_t}E^{\lambda}_n(\bu) &=
\nabla_{_{\scaleto{\bu_t(\bz_{\scaleto{0:t-1}{4pt}})}{6pt}}}c(\bu(\bz^n), \bz^n)\:  K_t(\bz^n_{0:t-1}, \cdot) + \lambda\bu_t,
\end{align}
where $\nabla_{_{\bu_t(\bz_{0:t-1})}}c\big(\bu(\bz^n), \bz^n\big)$ corresponds to the derivative of $c\big(\bu(\bz), \bz\big)$ with respect to its scalar arguments $u^1_t(\bz_{0:t-1}),\ldots, u^{r_t}_t(\bz_{0:t-1})$ evaluated at $\bz^n$:
\[ \nabla_{_{\scaleto{\bu_t(\bz_{\scaleto{0:t-1}{4pt}})}{6pt}}} c(\bu(\bz^n), \bz^n) = \left[\frac{\partial c(\bu(\bz^n), \bz^n)}{\partial u^1_t(\bz_{0:t-1})}, \ldots, \frac{\partial c(\bu(\bz^n), \bz^n)}{\partial u^{r_t}_t(\bz_{0:t-1})}\right].
\]
Thus, the update rule for the standard functional stochastic gradient descent (FSGD) algorithm becomes
\begin{align}
    \bu^{n+1}_t &= \!\bu_t^n - \eta_n \nabla_{\bu_t} E^\lambda_n(\bu^n)  \nonumber\\
    &= (1\!-\!\eta_n\lambda)\bu^{n}_t - \eta_n \nabla_{_{\scaleto{\bu_t(\bz_{\scaleto{0:t-1}{4pt}})}{6pt}}}c(\bu(\bz^n), \bz^n)\:  K_t(\bz^n_{0:t-1}, \cdot),\label{update_rule}
\end{align}
where $\eta_n$ is the step-size of the algorithm and the sequence of controllers is initialized at some fixed function $\bu_0\in \mathcal{H}$. 

Using the update rule in Eq. \eqref{update_rule}, we can easily show by induction on $n$ that if the initial decision is of the form $\bu_t^0(\cdot) = \bba^0_t\boldsymbol{K}_t(\bbd^0_t, \cdot)$ for some initial data matrix $\bbd^0$ and initial parameters $\bba_t^0$, then the solutions $\bu^n$ produced at every iteration also have this form. Specifically, for each $n>0$ and for all $t\in [T]$, there exist a scalar matrix $\bba^n_t$ and a data matrix $\bbd^n$ such that $
\bu^{n}_t(\cdot) = \bba^n_t \cdot \boldsymbol{K}_t(\bbd^n_t, \cdot)$. In fact, this parametrization allows us to rewrite the functional update rule in Eq. \eqref{update_rule} as a nonfunctional (scalar) update on the data matrix $\bbd^n$ and the parameters $\bba^n_1,\ldots, \bba^n_T$ as follows:
\[ \bbd^{n+1} = [\bbd^{n},\:\: \bz^n], \quad \bba^{n+1} =  \left[ (1 - \eta_n\lambda)\bba^{n},\: \: \eta_n \nabla_{\bu(\bz)}c\big(\bu^{n}(\bz^n), \bz^n\big) \right],\] 
where 
\[\bba^n\coloneqq \begin{bmatrix} \bba^n_1\\ \vdots\\ \bba^n_T\end{bmatrix},\quad \text{and}\quad  \nabla_{\bu(\bz)}c\big(\bu(\bz), \bz\big) \coloneqq \begin{bmatrix} \nabla_{\bu_1(\bz_0)}c\big(\bu(\bz), \bz\big)\\ \vdots \\ \nabla_{\bu_T(\bz_{0:T-1})}c\big(\bu(\bz), \bz\big) \end{bmatrix}. \] Notice that this update forces the data matrix to have one more column after every iteration, which brings us back to the same problem we had when applying the Representer Theorem. However, because this is an iterative algorithm, we will reduce the dimension of the data matrix $\bbd^n$ after every iteration by measuring the contribution of each individual observation $\bz^n$ and removing those observations that added almost no value to the decision.
\subsection{Proximal Projection}
We now describe how to reduce the number of observations in the data matrix $\bbd^n$ with the goal of reducing the dimension of the parameters $\bba^n$. We observed that the Representer Theorem as well as the FSGD algorithm generate decisions $u_{t, i}$ that belong to the subspace of $\mathcal{H}_t$ spanned by the functions $K_t(\bz^1_{0:t-1}, \cdot),\ldots, K_t(\bz^N_{0:t-1}, \cdot)$. What we want is to produce decisions that belong to a smaller subspace, one generated using fewer observations.

Suppose that  $\tilde{\bbd}^{n+1}$, and $\tilde{\bba}^{n+1}$ are the values resulting from the FSGD iterative rule in Eq. \eqref{update_rule}, i.e,
\[
 \tilde{\bbd}^{n+1} = [\bbd^{n},\: \: \bz^n] \quad \text{and} \quad \tilde{\bba}^{n+1} =   \left[ (1 - \eta_n\lambda)\bba^{n}, \:\: \eta_n \nabla_{\bu(\bz)}c(\bu^{n}(\bz^n), \bz^n) \right],
\]
which represent the decisions $\tilde{\bu}_t^{n+1}(\cdot) =\tilde{\bba}^{n+1}_t\boldsymbol{K}_t(\tilde{\bbd}^{n+1}_t, \cdot)$, and assume that we want to generate a decision that only uses observations from a smaller data matrix $\bbd^{n+1}$. We can approximate $\tilde{\bu}^{n+1}$ with a decision $\bu^{n+1}$ that only depends on observations in $\bbd^{n+1}$ by projecting each decision $\tilde{u}^{n+1}_{t, i}$ onto the subspace of $\mathcal{H}_t$ that is spanned by the functions $\boldsymbol{K}_t(\bbd^{n+1}_t, \cdot)$. If we denote this projection by  $\Pi_{\bbd^{n+1}}(\cdot)$ then we can define \begin{align}
\label{def_proj}
\bu^{n+1} \coloneqq \Pi_{\bbd^{n+1}}(\tilde{\bu}^{n+1}) = \Pi_{\bbd^{n+1}}\big((1-\eta_n\lambda)\bu^n - \eta_n\nabla_{\bu}c(\bu^n(\bz^n), \bz^n)\big).
\end{align} 
The projection operator can be computed by solving the least squares problem
\begin{align}
\label{LS}
\bba^{n+1} = \argmin_{\hat{\bba}^{n+1}} \: \: \sum_{t = 1}^T  \:\Big\|\tilde{\bba}^{n+1}_t\boldsymbol{K}_t(\tilde{\bbd}^{n+1}_t,\cdot ) - \hat{\bba}^{n+1}_t\boldsymbol{K}_t({\bbd}^{n+1}_t,\cdot )\Big\|^2_{\mathcal{H}^{r_t}_t},
\end{align}
which has a closed form solution given by
\begin{align}
    \bba^{n+1}_{t} = \left(\bbk_t[\bbd^{n+1}_t, \bbd^{n+1}_t]\right)^{-1} \bbk_t[\bbd^{n+1}_t, \tilde{\bbd}^{n+1}_t]\tilde{\bba}^{n+1}_t, \quad \text{for all } t \in [T].\label{LSsolution}
\end{align}
We then have a simple way to project the FSGD solution onto the Hilbert subspace generated by a smaller data matrix $\bbd^{n+1}$, but we are still left the question: how do we find the right data matrix $\bbd^{n+1}$? As in \cite{koppel2016parsimonious}, we use a method called destructive \textit{kernel orthogonal matching pursuit} (KOMP) with pre-fitting, which was developed in \cite{vincent2002kernel}. The KOMP algorithm takes as input a function $\tilde{\bu}\in\mathcal{H}$ (represented by its data matrix $\tilde{\bbd}$ as well as the corresponding parameters $\tilde{\bba}$), and a maximum error bound $\epsilon$. For each element $\bd$ in the data matrix $\tilde{\bbd}$, the algorithm computes the approximation function $\bu = \Pi_{\tilde{\bbd} \symbol{92} \{\bd\}}(\tilde{\bu})$ obtained by removing observation $\bd$ from $\tilde{\bbd}$. Next, the algorithm removes the observation that produced the lowest error, updates the current function accordingly and then repeats this procedure to remove the next element. The algorithm stops removing elements when the difference between the current function and the best approximation function is larger than $\epsilon$. The exact algorithm can be found in Algorithm $\ref{KOMP}$.

\begin{algorithm}[!htp]
\SetAlgoNoLine
\KwIn{Function $\tilde{\bu}$ represented by data matrix $\tilde{\bbd}$ with $\tilde{M}$ columns, parameters $\tilde{\bba}$, and $\epsilon>0$.}
 Initialize $\bbd = \tilde{\bbd}$, $M=\tilde{M}$, $\bba = \tilde{\bba}$, and $\bu = \tilde{\bu}$  \;
 \While{$\bbd$ is non-empty}{
 \For{$j = 1,\ldots, \tilde{M}$}{
  Find minimal approximation error with data matrix element $\bd^j$ removed:
  \begin{align*}
  \gamma^2_j &= \big\|{\bu} - \Pi_{\bbd \symbol{92} \{\bd^j\}}({\bu})\big\|^2_{\mathcal{H}} \\
  &=  \min_{\hat{\bba}} \:\sum_{t = 1}^T  \: \big\| {{\bba}}_t\cdot \bbk_t(\bbd_t, \cdot) -  \hat{\bba}_t \cdot \bbk_t(\bbd_t \symbol{92} \{\bd^j_t\}, \cdot)  \big\|^2_{\mathcal{H}_t}.
  \end{align*}}
  Find the index with minimum approximation error: $j^* = \argmin \gamma_j$ \\
  \eIf{ $\gamma_{j^*} >\epsilon$}{
   \textnormal{\textbf{stop}}\;
   }{
   Prune data matrix: $\bbd = \bbd \symbol{92} \{\bd^{j^*}\}$\;
   Update $M = M-1$\;
   Update the parameters:
    $\bba = \argmin_{\hat{\bba}}\: \:\sum_{t = 1}^T  \: \big\| {{\bba}}_t\cdot \bbk_t(\bbd_t, \cdot) -  \hat{\bba}_t \cdot \bbk_t(\bbd_t \symbol{92} \{\bd^j_t\}, \cdot)  \big\|^2_{\mathcal{H}_t}$.}
 }
 \KwOut{$\bbd$, $\bba$, $\bu$.}
 \caption{{Kernel Orthogonal Matching Pursuit (KOMP)} }
\label{KOMP}
\end{algorithm}
\subsection{The Algorithm}
By combining Functional Stochastic Gradient Descent with the Kernel Orthogonal Matching Pursuit we are able to develop an algorithm that approximates the minimizer of $\E^{\lambda}(\bu)$ with decision rules that are represented using only a few parameters. The algorithm is initialized with a decision rule $\bu^0_t = \bba^0 \boldsymbol{K}_t(\bbd^0, \cdot)$, which in practice is usually set to $0$. Then, in each iteration, it performs one FSGD step and then applies the KOMP algorithm in order to obtain an approximated decision with fewer observations. Notice that if we define the projected gradient $\tilde{\nabla}$ by
\begin{align}
\label{proj_grad}\tilde{\nabla}_{\bu}E^{\lambda}_n(\bu^n)  \coloneqq \frac{\bu^n - \Pi_{\bbd^{n+1}}[\bu^n - \eta_n\nabla_{\bu}E^{\lambda}_n(\bu^n)]}{\eta_n},
\end{align} then we can write the iterative updates of this procedure in the same form as the standard iterative updates of FSGD:
\begin{align}
\label{SMOK_step}\bu^{n+1} = \bu^{n} - \eta_n\tilde{\nabla}_{\bu}E^{\lambda}_{n}(\bu^{n}).
\end{align}
Since stochasticity does not guarantee a strict objective descent, the algorithm keeps track of the best decision rules observed and at the end it outputs the decision $\bu^*_\mathcal{S}$ with the lowest empirical error $E^{\lambda}_{\mathcal{S}}$ with respect to the data set $\mathcal{S}$. The exact formulation can be found in Algorithm \ref{SMOK}.\\

\begin{algorithm}[!htp]
\SetAlgoNoLine
\KwIn{ Data points $\mathcal{S} = \{\bz^n\}_{n = 1,\ldots, N}$, error bounds $\epsilon_n$, learning rate $\eta_n,$ and initial decision $\bu^0$ represented with data matrix $\bbd^0$ and parameters $\bba^0$.}
 \For{$n = 1,\ldots, N$}{
   \hspace{0.05 cm} Take FSGD step using the $n^{th}$ sample $\bz^n$ to obtain
  \[
    \tilde{\bbd}^{n+1} = [\bbd^{n}, \bz^n] \!\!\quad \text{and} \!\!\quad \tilde{\bba}^{n+1} \!= \!  \left[ (1 - \eta_n\lambda)\bba^{n}, \: \eta_n \nabla_{\bu(\bz)}c(\bu^{n}(\bz^n), \bz^n) \right].
\] Reduce the data matrix and number of parameters using 
\[\bbd^{n+1}, \bba^{n+1}, \bu^{n+1} = \textnormal{KOMP}(\tilde{\bbd}^{n+1}, \tilde{\bba}^{n+1}, \epsilon_n)\] 
  }
 \KwOut{$\bu^*_\mathcal{S} = \argmin_{\bu \in \{\bu^1,\ldots, \bu^N\}} \: E^{\lambda}_{\mathcal{S}}(\bu)$.}
 \caption{Sparse Multistage Optimization via Kernels [SMOK]}\label{SMOK}
\end{algorithm}

\section{Convergence Analysis}\label{sec:convergence}
In this section, we show that for a specific choice of step-size the objective value of the decision output by the algorithm converges to the objective value of the true minimizer.  We first present the three %DB you wrote three but you enumerated four % KIM fixed!
 main assumptions that we make on the problem settings in order to guarantee convergence of the algorithm:
\begin{assumption}
\label{assumption_kernel}
The data space $\mathcal{Z}$ is compact, the kernels $K_t$ are universal, and there exists a constant $\kappa$ such that 
\[
K_t(\bz_{1:t-1}, \bz_{1:t-1}) \leq \kappa, \quad \forall \: \bz \in \mathcal{Z},\quad \forall \: \: t\in [T].
\]
\end{assumption}
\begin{assumption}
\label{assumption_lip}
There exists a constant $C$ such that {for all $\boldsymbol{z}\in \mathcal{Z}$} the loss function satisfies \[
\big|c(\bu, \bz) -c(\bu', \bz)\big|  \leq C \| \bu - \bu'\|_2, \quad \forall \: \: \bu, \bu'\in \mathbb{R}^{r_1 + \hdots + r_T}.
\]
\end{assumption}
\begin{assumption}
\label{assumption_convex}
The loss function $c(\bu(\bz), \bz)$ is convex and differentiable with respect to the scalar arguments $\bu(\bz)$ for all $\bz\in \mathcal{Z}$.
\end{assumption}

Assumption $\ref{assumption_kernel}$ naturally holds for most data domains, and this is a necessary assumption to ensure that the Hilbert norm of the optimizer of $E^{\lambda}$ is bounded. Assumption \ref{assumption_lip} holds whenever the cost function $c$ as well as the constraint functions $g_q$ are Lipschitz. This assumption implies that the gradient of $c$ with respect to the scalars $\bu(\bz)$ is bounded as 
\begin{align}
\label{gradient bound}
\|\nabla_{\bu(\bz)} c(\bu(\bz), \bz)\|_2\leq C,
\end{align} which in turn allows us to upper bound the expected norm of the gradient $\Exp\left[\|\nabla_{\bu}E^\lambda_n(\bu)\|^2_{\mathcal{H}_t}\right]$. Assumption \ref{assumption_convex} is a standard condition for convergence of descent methods, and it can be relaxed to the case in which the loss function is almost everywhere differentiable by applying subgradients instead of gradients. 
\begin{theorem}\label{convergence-rates}
Let $\bu_{\mathcal{S}}^* \coloneqq \argmin_{\bu \in \{\bu^1, \hdots, \bu^N\}} E_{\mathcal{S}}^{\lambda}(\bu)$ be the decisions generated by Algorithm \ref{SMOK} when given the set $\mathcal{S} = \{\bz^n\}_{n = 1}^N$ as input, and let $\bu^{\lambda}$ be the true minimizer of $E^{\lambda}(\bu)$ over $\mathcal{H}$. If we use constant step-size $\eta$ %with $\eta = \frac{P_1}{\sqrt{N}}<\frac{1}{\lambda}$, and $P_1 >0$,
and constant error bounds $\epsilon = P_2\eta^2$ for some constant $P_2 >0$, then under Assumptions 1-3,  we have that
\[
\Exp\left[E^{\lambda}(\bu_{\mathcal{S}}^*)  - E^{\lambda}(\bu^{\lambda})\right] \leq  \mathcal{O}\left(\frac{\eta}{\lambda}\right).
\]
 %DB   Having the forall quantifier here makes no sense. N is teh data size, no? %KIM fixed this, N is indeed fixed. 
\end{theorem}
\begin{proof}
See Appendix \ref{appendix:theorems}. \hfill
\end{proof}
\begin{corollary}\label{corollary}
Let $\bu^*$ be the true minimizer of $E(\cdot)$ over $\mathcal{F}$. If we use constant step-size with $\eta = \frac{P_1}{\sqrt{N}}<\frac{1}{\lambda}$, and $P_1 >0$, constant error bounds $\epsilon = P_2\eta^2$ for some constant $P_2 >0$, and regularization parameter $\lambda$ such that $\lambda \xrightarrow[N\to\infty]{} 0$ and $\lambda \sqrt{N} \xrightarrow[N\to \infty]{}\infty$, then under Assumptions 1-3  we have that
\begin{align}
\lim_{N\rightarrow\infty} \mathbb{E}[|E(\bu^*_\mathcal{S}) - E(\bu^*)|] = 0.
\end{align}
%Since $L_1$ convergence implies convergence in probability, we also obtain that E^{\psi}(\bu^*_\mathcal{S}) \xrightarrow[\mathbb{P}]{}E^{\psi}(\bu^{\psi})$.
\end{corollary}
\begin{proof}
See Appendix \ref{appendix:theorems}. \hfill
\end{proof}
Since $L_1$ convergence implies convergence in probability, the corollary also implies that the expected loss achieved with Algorithm \ref{SMOK} converges in probability to the optimal solution.

In addition, from Theorem $\ref{convergence-rates}$ we observe that setting $\eta = \frac{P_1}{\sqrt{N}}$ makes the objective value of the solution found by Algorithm \ref{SMOK} converge to the optimal solution of problem \eqref{directemp} with a rate of convergence of $\mathcal{O}\left(\frac{1}{\lambda \sqrt{N}}\right)$. Convergence can also be achieved under diminishing step size, although with a slower rate of  $\mathcal{O}\left(\frac{1}{\lambda \log{N}}\right)$. In practice, a diminishing step size or a very small constant step size might make our data matrix $\bbd^n$ grow arbitrarily large, since little or no pruning would be done at each iteration. A constant step size is then what allows us to control the trade-off
between accuracy and memory required; we want to use a step size $\eta$ that is small enough to make the error in Theorem \ref{convergence-rates} small, but large enough for the pruning to be done.
\section{Complexity Analysis}\label{sec:complexity}
Let $M_n$ be the size of the data matrix $\bbd^n$ during the $n^{th}$ iteration of Algorithm \ref{SMOK}. We analyze both space and time complexities per iteration in terms of $M_n$.\\

\noindent \textit{\textbf{Space:}} At each iteration we need to store the kernel matrix $\bbk_t[\bbd^n_t, \bbd^n_t]\in \mathbb{R}^{M_n\times M_n}$ and its inverse as well as the parameters $\bba_t^{n}\in \mathbb{R}^{r_t\times M_n}$ for each $t$. This results in $\mathcal{O}(TM_n^2 + M_n\sum_{t = 1}^T r_t)$ memory requirement.

\noindent \textit{\textbf{Time:}} For the FSGD step, computing the gradient takes $\mathcal{O}(M_{n}\sum_{t = 1}^T r_t)$ time. Computing from scratch the kernel matrices $\bbk_t[\bbd^n_t, \bbd^n_t]\in \mathbb{R}^{M_n\times M_n}$ and its inverses (needed for the pruning step) takes $\mathcal{O}(M_n^2\sum_{t = 0}^T q_t)$ and $\mathcal{O}(TM_n^3)$ time respectively. However, by using a recursive rule to compute these matrices in terms of the corresponding values in the previous iteration, the times become $\mathcal{O}(M_n\sum_{t = 0}^T q_t)$ and $\mathcal{O}(M_n^2)$ respectively. In addition, the matrix multiplication in Eq. \eqref{LSsolution} takes $O(M_n^2)$ time for each $t$. Since at most $M_{n}$ elements can be removed from the dictionary at the $n^{th}$ iteration, we obtain that in the worst case scenario the time per iteration becomes  $\mathcal{O}(TM_n^3 + M_n^2\sum_{t = 0}^T q_t)$.\\

Let us now discuss the size of $M_n$. In the
worst-case, we know that for all iterations the size of the data matrix is upper bounded by the covering number $M$ of the data domain \citep{ZHOU2002739}. More specifically, for fixed step size $\eta$ and fixed error bound $\epsilon = P_2\eta^2$, we have that if the data space $\mathcal{Z}$ is compact (Assumption \ref{assumption_kernel}), then $M_n$ is upper bounded by the minimum number of balls of radius $\frac{P_2\eta}{C}$ needed to cover the compact set $K_1(\mathcal{Z}_0, \cdot)\times \hdots\times K_T(\mathcal{Z}_{0:T-1}, \cdot)$ of kernel transformations (see for example the proof of Theorem 3 in \cite{koppel2016parsimonious}). While an exact expression for this cover number $M$  is unknown, the number is finite \citep{anthony2009neural} and it decreases as $\eta$ or $P_2$ increases. In particular, the maximum number of samples in the data matrix depends on the step size $\eta$ and the constant $P_2$, but not on the data size $N$. 

Denoting the cover number described above by $M$ and considering fixed values of $T$ and of the dimensions $r_1,\hdots, r_T$ and $q_0,\hdots, q_T$, we obtain that the worst case total time across the $N$ iterations of Algorithm \ref{SMOK} can be upper bounded by $\mathcal{O}(NM^3)$ and worst case total space required is $\mathcal{O}(NM^2)$. While the worst case scenario cannot happen for all iterations (for example, if $M$ elements are pruned in one iteration, the next iteration is very fast), this bound is enough to conclude that total time and total space are in the worst case linear in the number of iterations. Notice that if we removed the pruning step, the entire algorithm would require $\Omega(N^2)$ space to store the kernel matrix and $\Omega(N^2)$ time for computations, showing that Algorithm \ref{SMOK} indeed reduces the overall complexity as the number of iterations becomes much larger than $M$.
\section{Computational Experiments}\label{sec:experiments}
We perform computational experiments for the inventory control and the shipment planning problems to analyze the average out-of-sample performance as well as the tractability of the proposed algorithms. For both applications we compare the SMOK algorithm proposed in Algorithm \ref{SMOK}, to the MOK algorithm (Multistage Optimization with Kernels), which is the result of applying the FSGD algorithm without the pruning step. Moreover, we compare the SMOK and MOK algorithms against three other benchmarks:
\begin{enumerate}
\item \textbf{SRO:} Sample robust optimization approach from \cite{bertsimas2018data}, in which all samples are assigned equal weight $\frac{1}{N}$. We use uncertainty sets bounded by $\epsilon$ in the $\ell_1$ norm as well as multi-policy approximation with linear decision rules. 
    \item \textbf{SRO-knn:} Sample robust optimization with covariates approach developed in \cite{bertsimas2019predictions}, using uncertainty sets bounded by $\epsilon$ in the $\ell_1$ norm as well as multi-policy approximation with linear decision rules. The weights were obtained using the $k_N$-nearest neighbors approach.
    \item \textbf{SAA-knn:} Sample average approximation method, which is equivalent to the SRO-knn approach with ($\epsilon = 0$).
\end{enumerate}

{We analyze the computational results for several instances of the inventory control problem. First, we consider a high dimensional instance of the problem to show the tractability of the SMOK algorithm as well as to compare its performance against other methods.} Next, we analyze how the performance of the proposed algorithms varies with the dimensions of the problem (number of periods, data size, dimension of the data as well as dimension of the controllers). For instances in which the number of periods is less than 5 we are also able to compute lower bounds for the loss achieved by the optimal decision rules, which enables us to quantify the optimality gap of the proposed methods.

For the shipment planning application we reproduce the results from \cite{bertsimas2019predictions} to compare the SMOK and MOK algorithms against sample robust optimization (with and without covariates) and sample average approximation. For training all these benchmarks we use the same parameter values reported in \cite{bertsimas2019predictions}.\\

\noindent \textbf{Handling Constraints:} Often the sequence of decisions $\bu(\bz)$ must satisfy certain convex constraints for all possible disturbances, transforming the problem of interest into
\begin{equation}
\begin{array}{rl}
\displaystyle\min_{\bu \in \mathcal{F}} \quad \:  & \displaystyle \Exp_{\bz} \big[ c\big(\bu(\bz), \bz\big) \big]  \vspace{3pt} \\
\text{s.t.} \quad \: &  g_q\big(\bu(\bz)\big) \leq 0,  \quad \forall \: \bz \in \mathcal{Z}, \quad \forall \: q\in [Q].
\end{array}\label{constrained_version}
\end{equation}
We address this problem by relaxing the constraints into the objective with a penalty function. More specifically, in Algorithm \ref{SMOK} we replace the cost $ c(\bu(\bz), \bz)$ with a new loss function $c^\psi$ defined as
\begin{align}
c^\psi\big(\bu(\bz), \bz\big) &\coloneqq  c\big(\bu(\bz), \bz\big) + \psi\sum_{q = 1}^Q\max\Big(0, g_q\big(\bu(\bz)\big)\Big)^2, \label{augmented_loss}
\end{align}
where $\psi$ is the penalty parameter. Although feasibility is not guaranteed, the constraint violation is expected to vanish for large enough $\psi$ (see Lemma \ref{const-violation}). Convergence analysis for the SMOK algorithm applied to this constrained problem can be found in Appendix \ref{appendix:theorems}.\\

\noindent \textbf{Parameter Settings:} We train the SMOK and MOK algorithms using Gaussian kernels and constant step size. The values for $\lambda, \psi$ and $\theta$ were found using validation, and the decisions were projected onto the space of feasible decisions before making any evaluations, both at training and testing stages (this means that the decisions evaluated had $0$ constraint violation). For each instance of the problem the constant step size $\eta$ was initially set to $10^{-5}$ and it was repeatedly increased by factors of 5 so long as the average training loss did not worsen and the iterations were reaching convergence. The parameter $P_2$ for the error bound $\epsilon$ was initially set to $0.1$ and was repeatedly increased by factors of 2; we stopped increasing it when the average training loss significantly worsened. \\%DB How exactly is this achieved? %KIM I now explain that we control this through the parameter P_2.

\noindent \textbf{Software Utilized:}{ Experiments were implemented in Python 3 \citep{10.5555/1593511} using the NumPy library \citep{harris2020array}. We clarify that Eq. \eqref{LSsolution} can often be difficult to compute due to numerical instability in the calculations for the inverse matrix. To address this issue we add a small value $\lambda = 1e^{-7}$ to the diagonal of a matrix before computing its inverse. In terms of hardware, all experiments where run on an Intel(R) Core(TM) i7-8557U CPU @ 1.70GHz processor with 4 physical cores (hyper-threading enabled). The machine has a 32KB L1 cache and 256KB L2 cache per core, and an 8MB L3 cache. There is a total of 16GB DRAM.}
\subsection{Inventory Control Problem}
We consider a multistage inventory control problem with linear constraints.  At each stage $t$ with initial inventory $s_t$, a retailer places procurement orders $\bu_t\in \mathbb{R}^r$ at various suppliers, and later observes the demands $\bw_t\in \mathbb{R}^q$. At the end of each stage, the firm incurs a per-unit holding cost of $h_t$ and a back-order cost of $b_t$. The inventory is not backlogged, and therefore the initial inventory for the next period is given by the linear equation $s_{t} = s_{t-1} + \boldsymbol{1}^\top \bu_t - \boldsymbol{1}^\top \bw_t$, with zero initial inventory for the first period. In addition, the procurement orders are upper bounded by a constant $L$ and the sum of procurement orders for two consecutive stages cannot exceed a constant $\ell$. As in Ban et al. (2018), we consider the scenario in which retailers can observe auxiliary covariates $\bx$ that relate to the future demands (e.g. in the fashion industry color and brand are useful factors for predicting demand of the products). For a problem with $T$ periods, we can formulate this optimization problem as
\begin{align*}
    \min_{\bu_{1:T}} \quad & \mathbb{E}_{ \bw | \bx}\left[\sum_{t = 1}^T h_t\left[s_t\right]^+ + b_t \left[-s_t\right]^+ \bigg \vert \: \bx = \bx_0   \right]\\
    \text{s.t.}\quad  & s_t = s_{t-1} + \boldsymbol{1}^\top\bu_t - \boldsymbol{1}^\top\bw_t,& &\forall t\in [T],\\
    & \bu_t \geq \boldsymbol{0},& &\forall t\in [T],\\
    & \bu_t \leq L\boldsymbol{1},& &\forall t\in [T],\\
    & \bu_t + \bu_{t+1} \leq \ell\boldsymbol{1},& &\forall t\in [T-1].
\end{align*}
The parameters $h_t, b_t$ were chosen to be $2$ and $1$, respectively. The data sets used in these experiments were generated by sampling $\bx$ from a Truncated Gaussian Distribution with mean $2$ and standard deviation $0.5$, and with truncating bounds $0$ and $6$. The demands $\bw_t$ were then obtained as a linear function of the covariates with some added noise; specifically, $\bw_t = \alpha_t\bx + \boldsymbol{\epsilon}_t$, where $\epsilon_t$ was sampled from a standard distribution and the constants $\alpha_t$ were selected to be close to $50$.

{We first consider a large instance of the problem with $T = q = r = 10$, and we set the control bounds as $L = 150$ and $\ell = 200$. We use a training set with $2000$ sample paths and we approximate the expected loss achieved by each method by averaging the losses across a common testing set with $10^4$ sample paths. Since the SRO and SRO-knn methods become intractable for problems of this magnitude, in this experiment we only compare the SMOK and MOK methods to SAA-knn. We use validation to choose the best parameters for all methods and we evaluate the results on the testing set. In table \ref{table:large-instance} we observe that both SMOK and MOK outperform SAA-knn in terms of average out-of-sample loss and computational time. Moreover, the number of parameters needed for the SMOK algorithm is smaller by two orders of magnitude compared to the other methods. Even though we observe an increase in computation time for SMOK with respect to MOK (due to the overhead computation time for the pruning step), we also see that adding sparsity helped SMOK achieve a better average loss.
\begin{table}[!htp]
    \centering
    \begin{tabular}{|c|c|c|c|}
    \hline
          & \textbf{Avg OOS Loss} & \textbf{Total Time (hours)} &\textbf{No. of Params}\\
         \hline
         \hline
         \textbf{SMOK}& $491.30$ & $0.3$ & $1.5\times 10^3$\\
         \hline
         \textbf{MOK}& $493.74$ & $0.1$ & $ 2\times 10^5$\\
         \hline
         \textbf{SAA-knn} & $496.04$ & $14.36$ &  $ 2.2\times 10^5$\\
         \hline
    \end{tabular}
    \caption{Average out-of-sample (OOS) loss and total computation time for inventory problem with $T = q = r = 10$.}
    \label{table:large-instance}
\end{table}}

We next consider other instances of the inventory problem to analyze how the dimensions of the problem affect the overall performance of the SMOK and MOK algorithms. We compared these two methods to a third algorithm ADR (Affine Decision Rules), which refers to the common approximation technique of restricting the space of decision rules to be affine functions. We train all methods using the same training sets and the same validation sets (with size equal to $30\%$ of the training size), and we approximate the expected loss achieved by averaging across a common testing set of $10^5$ sample paths. In addition, we compute lower bounds for the optimal expected loss when $T\leq 5$ (see Appendix \ref{lower_bound}), which allows us to analyze the optimality gap for the different methods.

Multiple data sets were generated to analyze the performance of the algorithms as we increase the number of periods, the training size, the dimension of the data and the dimension of the controls. In each case we analyze the average out-of-sample loss  and the size $M$ of the data matrix, which refers to number of parameters per control. We also analyze the computational time for each iteration of Stochastic Gradient Descent (projected or not projected), and the evaluation time (time it takes to evaluate the empirical loss function $E_{\mathcal{S}}^{\lambda}(\bu)$ given the parameters for the functional representation of $\bu$). Notice that since the stochastic gradient descent algorithm does not strictly descend, the empirical loss of the validation set needs to be evaluated every certain number of iterations, which makes the evaluation time part of the total training time.
\subsubsection{Varying the Number of Periods: ($L = 150, \ell =200, q = r = 1, N=2000, T=2,3,4,5$)}

In Figure \ref{fig:num_periods_gd}, we observe that the convergence trajectory is not significantly affected by the pruning step, and the number of iterations needed until convergence does not change much for $T\geq 3$. In addition, we see in Figure \ref{fig:num_periods_lowbound} that ADR results in very poor performance, while both the SMOK and MOK algorithms are quite close to the lower bounds found for the optimal expected loss. In Figure \ref{fig:num_periods_iter} we observe that the time per iteration of stochastic gradient descent grows linearly for both SMOK and MOK, but MOK takes longer times due to the overhead introduced by the pruning step. The evaluation time (Figure \ref{fig:num_periods_eval}) also grows linearly for both algorithms, although unlike the time per iteration, the slope is larger for MOK than for SMOK because the number of parameters is significantly smaller for this last method (SMOK algorithm reduced the size $M$ of the data matrix from $2000$ to values below $15$).\\

\begin{figure}[!htp]
     \centering
     \begin{subfigure}[b]{0.4\textwidth}
         \centering
         \includegraphics[trim=0 0 0 20,clip,width=\textwidth]{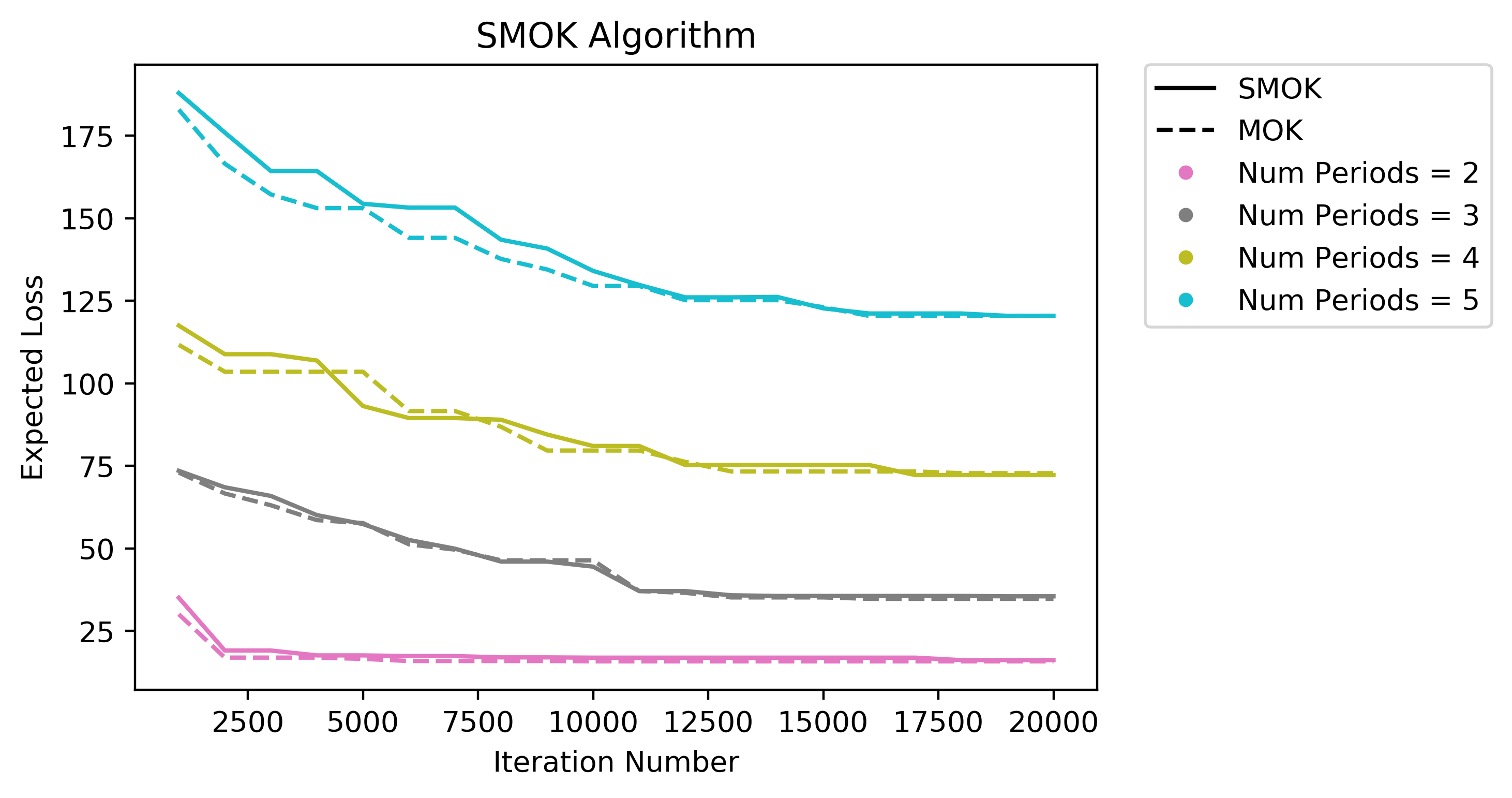}
         \caption{}
         \label{fig:num_periods_gd}
     \end{subfigure}
     \begin{subfigure}[b]{0.4\textwidth}
         \centering
         \includegraphics[trim=0 0 0 20,clip,width=\textwidth]{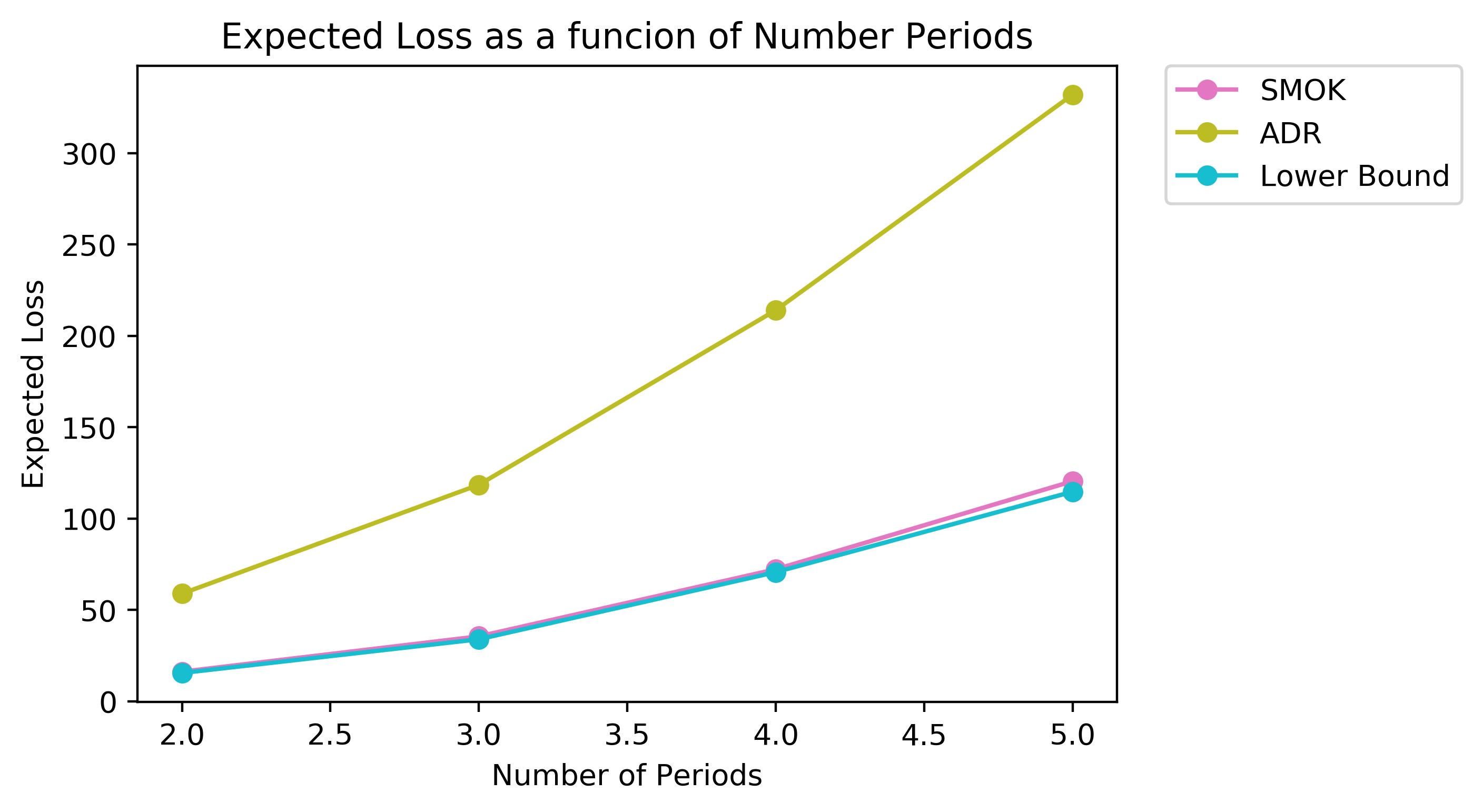}
         \caption{}
         \label{fig:num_periods_lowbound}
     \end{subfigure}\\
     \begin{subfigure}[b]{0.4\textwidth}
         \centering
         \includegraphics[ width=\textwidth]{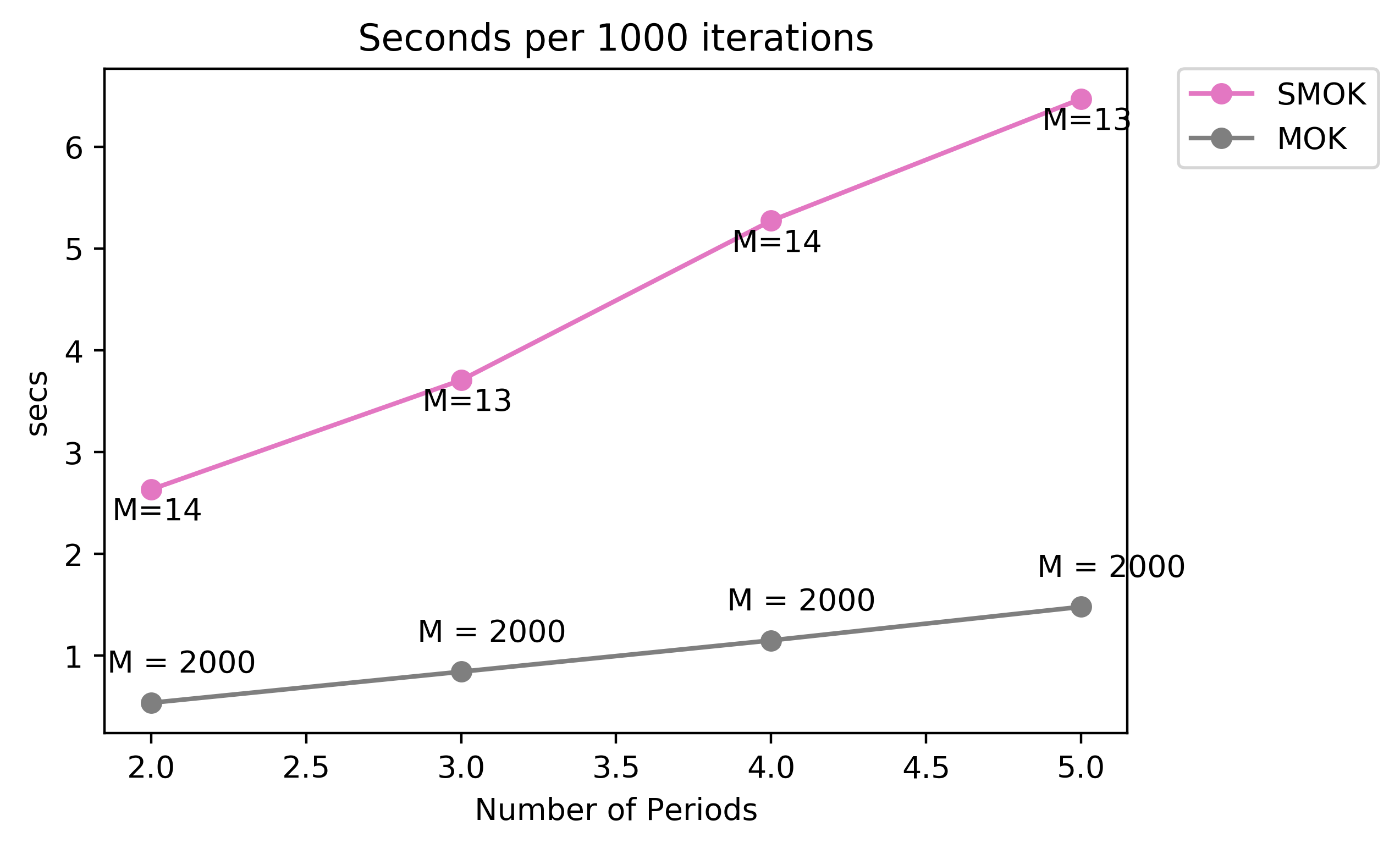}
         \caption{}
         \label{fig:num_periods_iter}
     \end{subfigure}
     \begin{subfigure}[b]{0.4\textwidth}
         \centering
         \includegraphics[width=\textwidth]{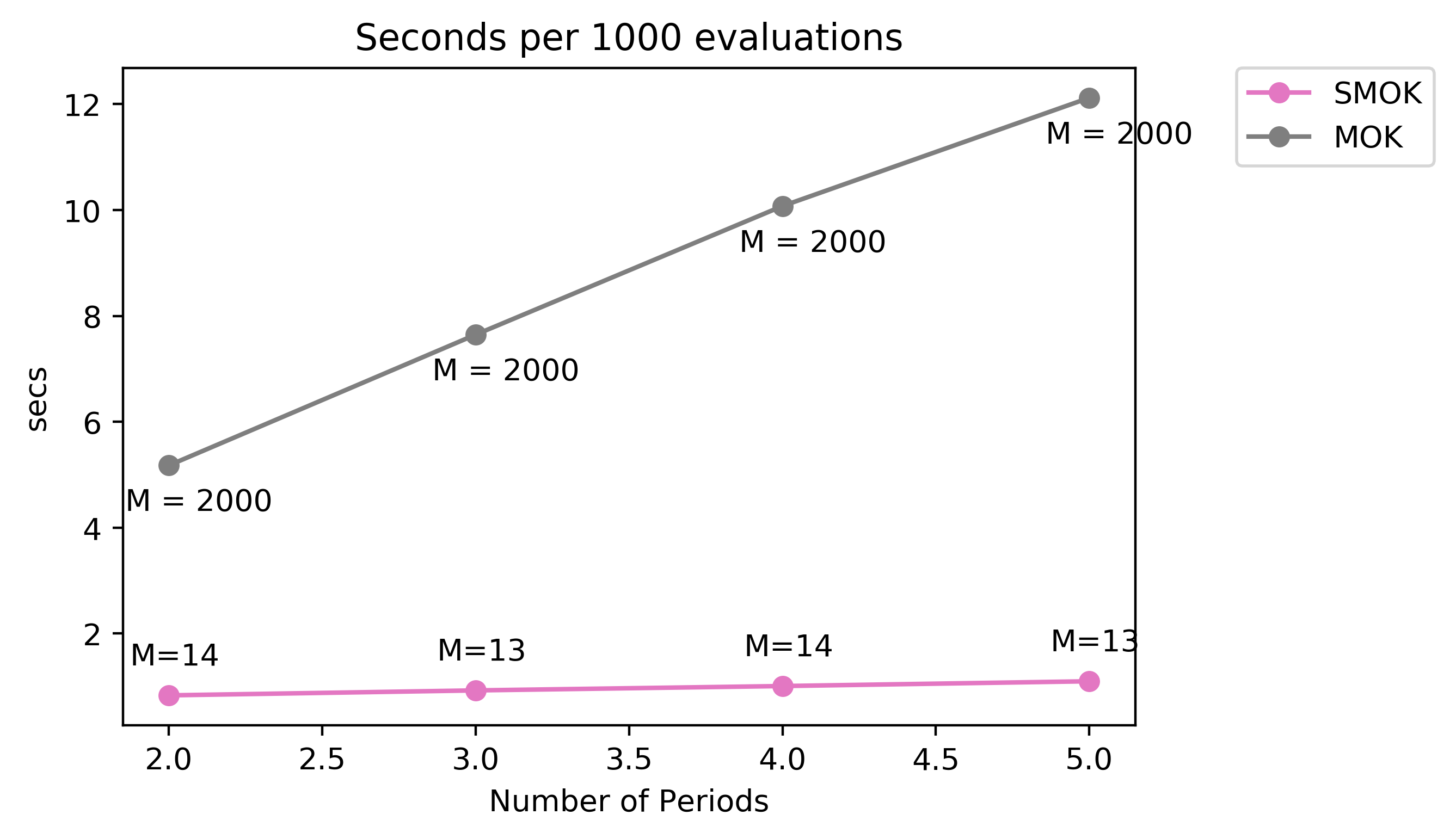}
         \caption{}
         \label{fig:num_periods_eval}
     \end{subfigure}
        \caption{Expected loss and computational time for varying number of periods.}
        \label{fig:num_periods_loss}
\end{figure}

\subsubsection{Varying the Data Size: ($L = 150, \ell =200, q = r = 1, T=3, N=10,100,1000,4000,7000,10000$)} 

Figure \ref{fig:data_size_lowbound} shows that, as anticipated, the expected loss achieved by both MOK and SMOK algorithms decreases as the size of the training set becomes larger. The number of iterations required to reach convergence (Figure \ref{fig:data_size_gd}) does not change much with the data size and the expected loss achieved remains relatively constant after a large enough training size, which occurs around $N=1000$. In Figures \ref{fig:data_size_iter},\ref{fig:data_size_eval} we can observe a significant memory improvement of SMOK over MOK when $N$ becomes very large. For $N=10^4$, for example, SMOK outputs decision rules with only $11$ parameters, while SMOK requires $10^4$ parameters per control. The evaluation time in Figure \ref{fig:data_size_eval} grows quadratically with the number of parameters in each control (the quadratic factor comes from computing the kernel matrix $\bbk_t[\bbd_t, \bbd_t]$), which in the case of MOK corresponds to the size of the training set. Since the SMOK algorithm has much fewer parameters, it takes under half a second to evaluate the average loss of $1000$ samples regardless of the training data size. Notice that the time per iteration (Figure \ref{fig:data_size_iter}) is higher for SMOK than for MOK when $N$ is small due to the pruning step. However, we observe that the time per iteration increases linearly for MOK while it stabilizes for SMOK, implying that for bigger values of $N$ the SMOK method actually takes less time per iteration and per evaluation. 
\begin{figure}[!htp]
     \centering
     \begin{subfigure}[b]{0.4\textwidth}
         \centering
         \includegraphics[trim=0 0 0 20,clip,width=\textwidth]{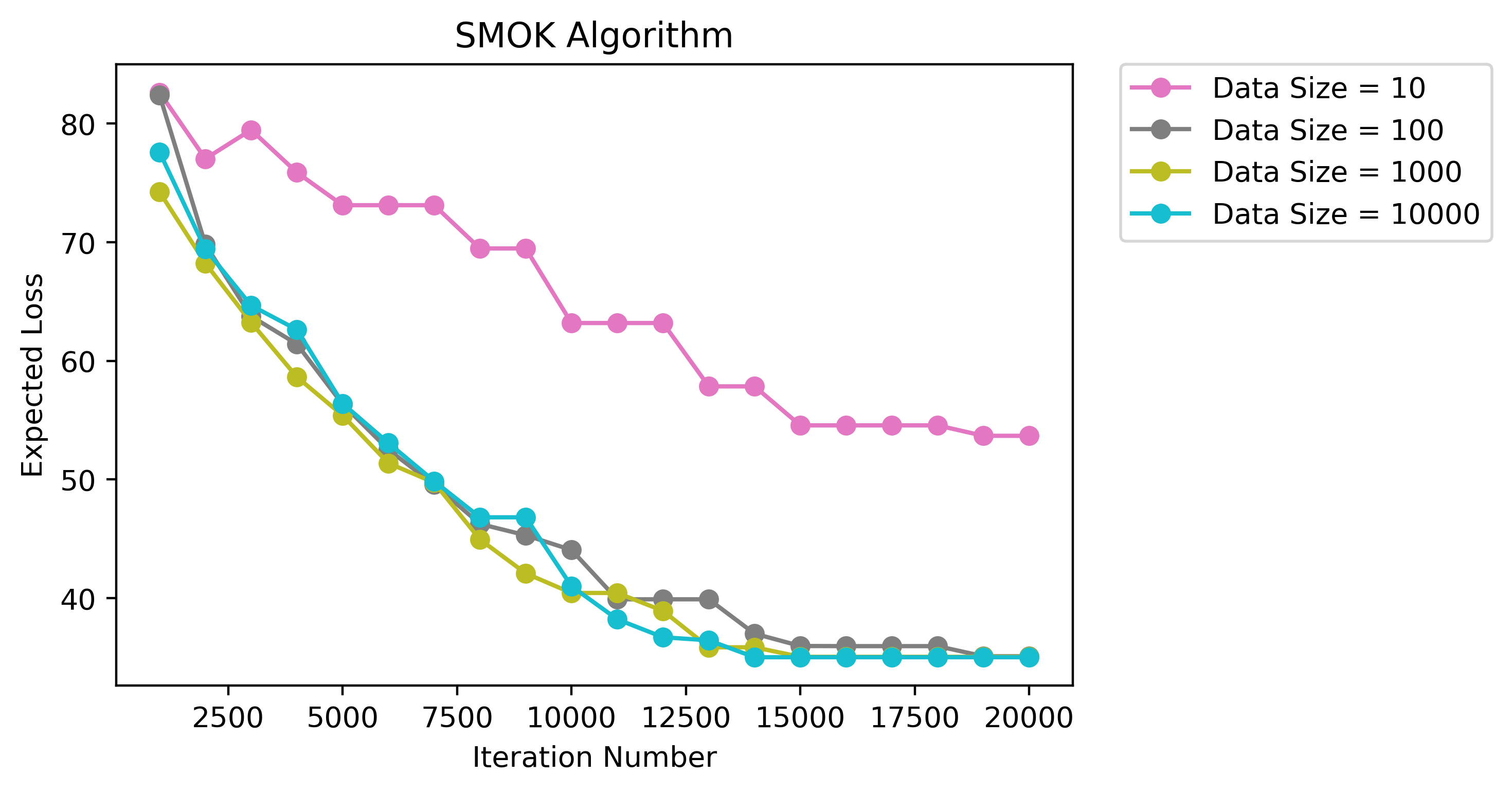}
         \caption{}
         \label{fig:data_size_gd}
     \end{subfigure}
     \begin{subfigure}[b]{0.4\textwidth}
         \centering
         \includegraphics[trim=0 0 0 20,clip,width=\textwidth]{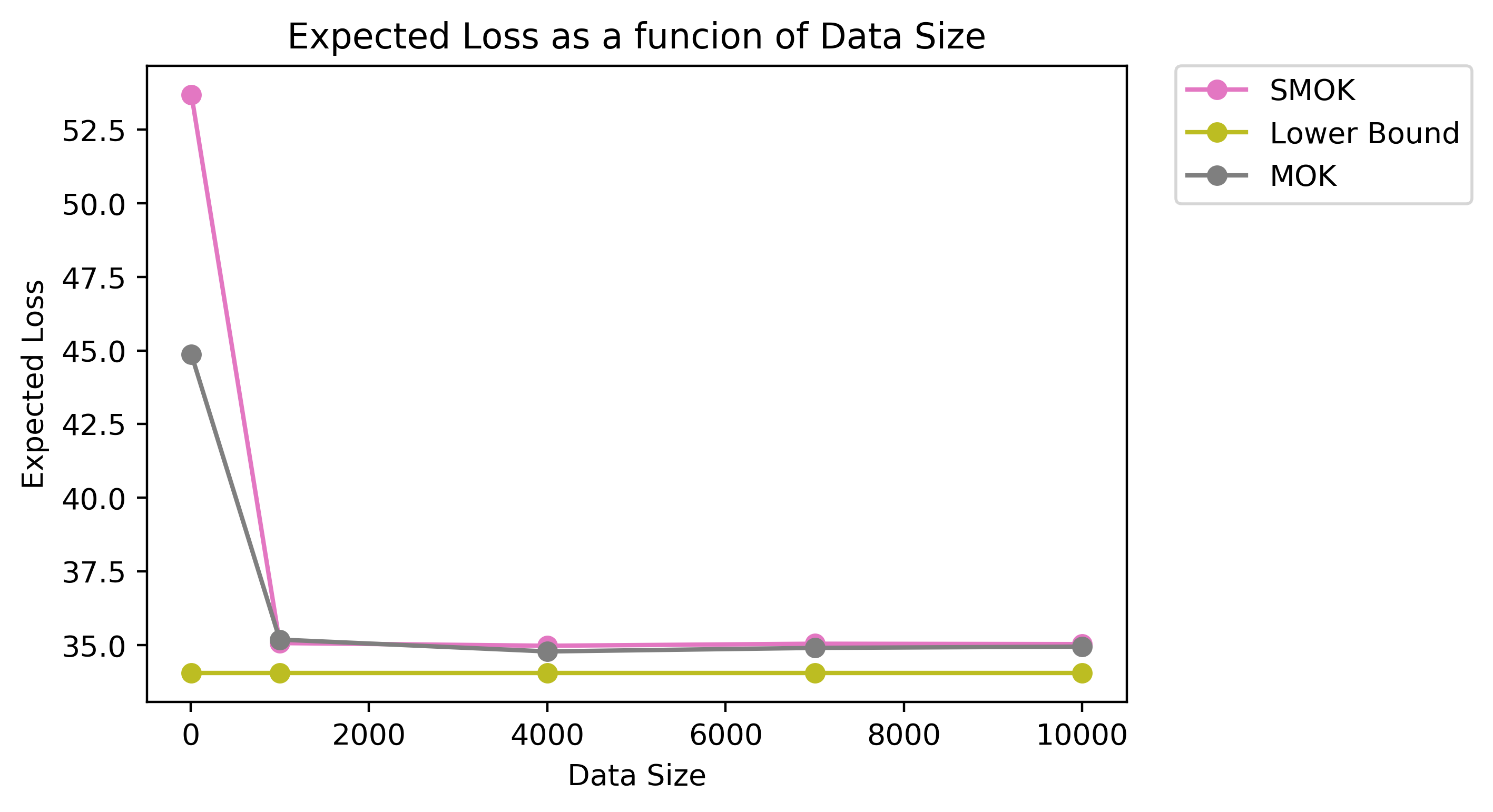}
         \caption{}
         \label{fig:data_size_lowbound}
    \end{subfigure}
        \begin{subfigure}[b]{0.4\textwidth}
         \centering
         \includegraphics[width=\textwidth]{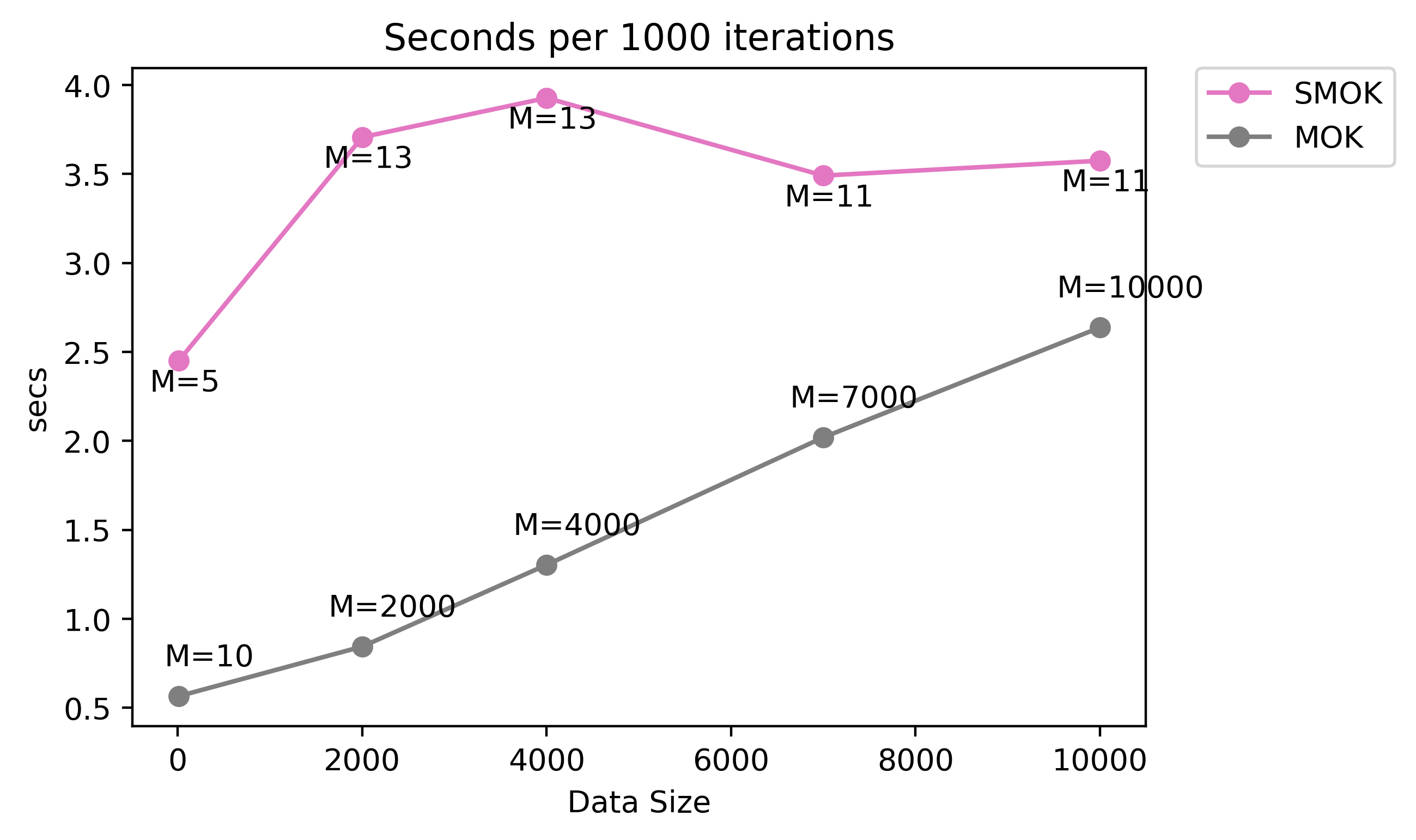}
         \caption{}
         \label{fig:data_size_iter}
     \end{subfigure}
     \begin{subfigure}[b]{0.4\textwidth}
         \centering
         \includegraphics[width=\textwidth]{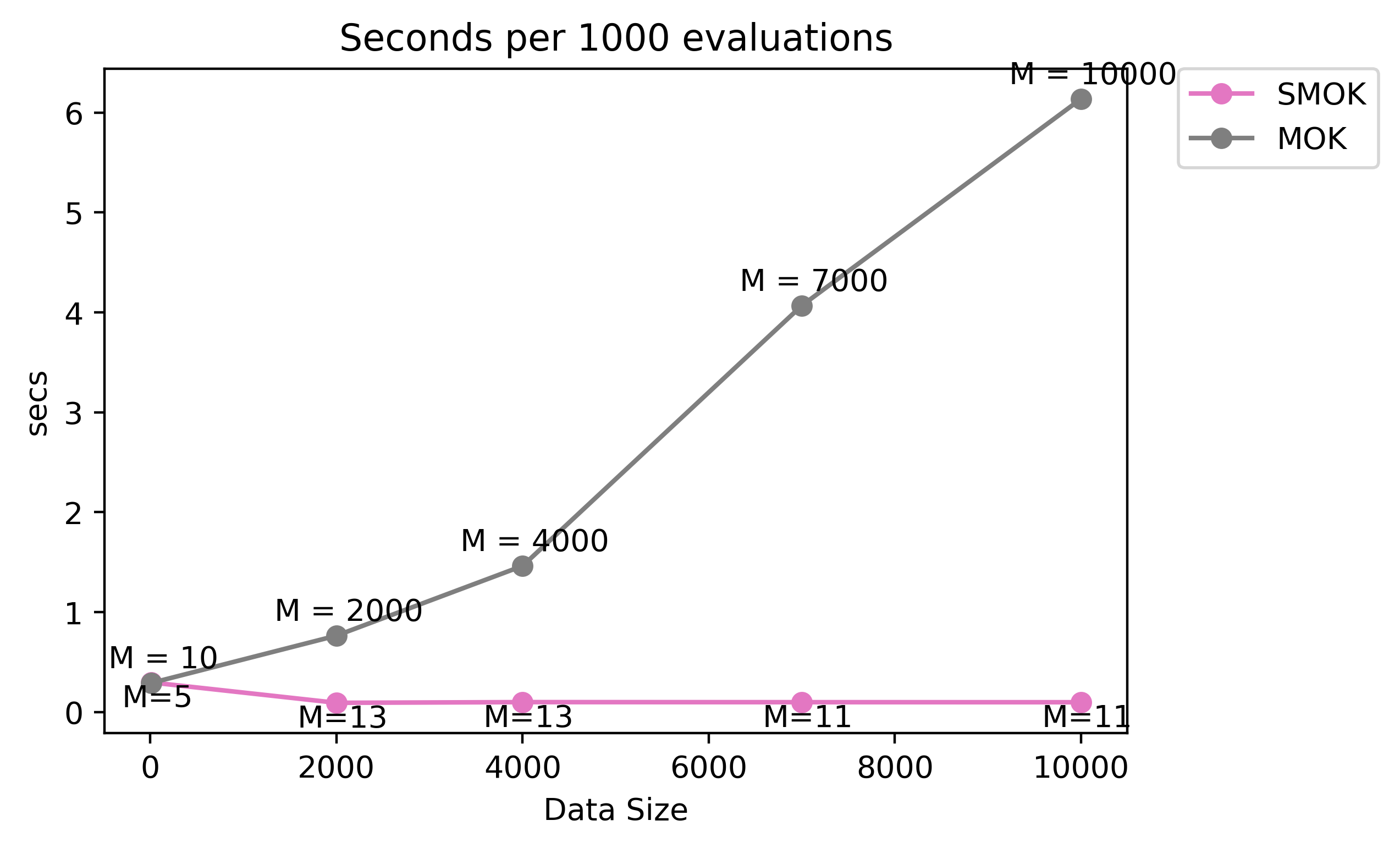}
         \caption{}
         \label{fig:data_size_eval}
     \end{subfigure}
        \caption{Expected loss and computational time for varying data sizes.}
        \label{fig:data_size_loss}
\end{figure}
\subsubsection{Varying Data Dimension: ($L = 150, \ell =200, r = 1, T=3, N=2000, q = 1, 10, 20, 30, 40, 50$)}  When generating data sets for this part we enforce that the value $\sum_q{(\bw_t)_q}$ remains constant for all $t\in [T]$, which guarantees that the optimal expected loss is the same across instances. In Figure \ref{fig:data_dim_gd}, we observe that the trajectories of the expected loss across the FSGD iterations are quite similar for all the different dimensions of the data. More importantly, the error gap does not worsen as the dimension of the data increases (Figure \ref{fig:data_dim_lowbound}), showing that the accuracy of our algorithms does not worsen for data sets in large dimensional spaces. Additionally, in Figure \ref{fig:data_dim_eval} we observe that there is a slight linear increase in the evaluation time for both SMOK and MOK algorithms, which is expected since the dimension of the demand vector affects the computation of the exponent in the Gaussian kernel. In terms of the iteration time (Figure \ref{fig:data_dim_iter}), we can see that SMOK remains quite stable around $4$ seconds per $1000$ iterations, while MOK shows linear increase. As in the previous examples, the number of parameters of the SMOK algorithm is quite similar across the different experiments and remains under $15$.
\begin{figure}[!htp]
     \centering
     \begin{subfigure}[b]{0.4\textwidth}
         \centering
         \includegraphics[trim=0 0 0 20,clip,width=\textwidth]{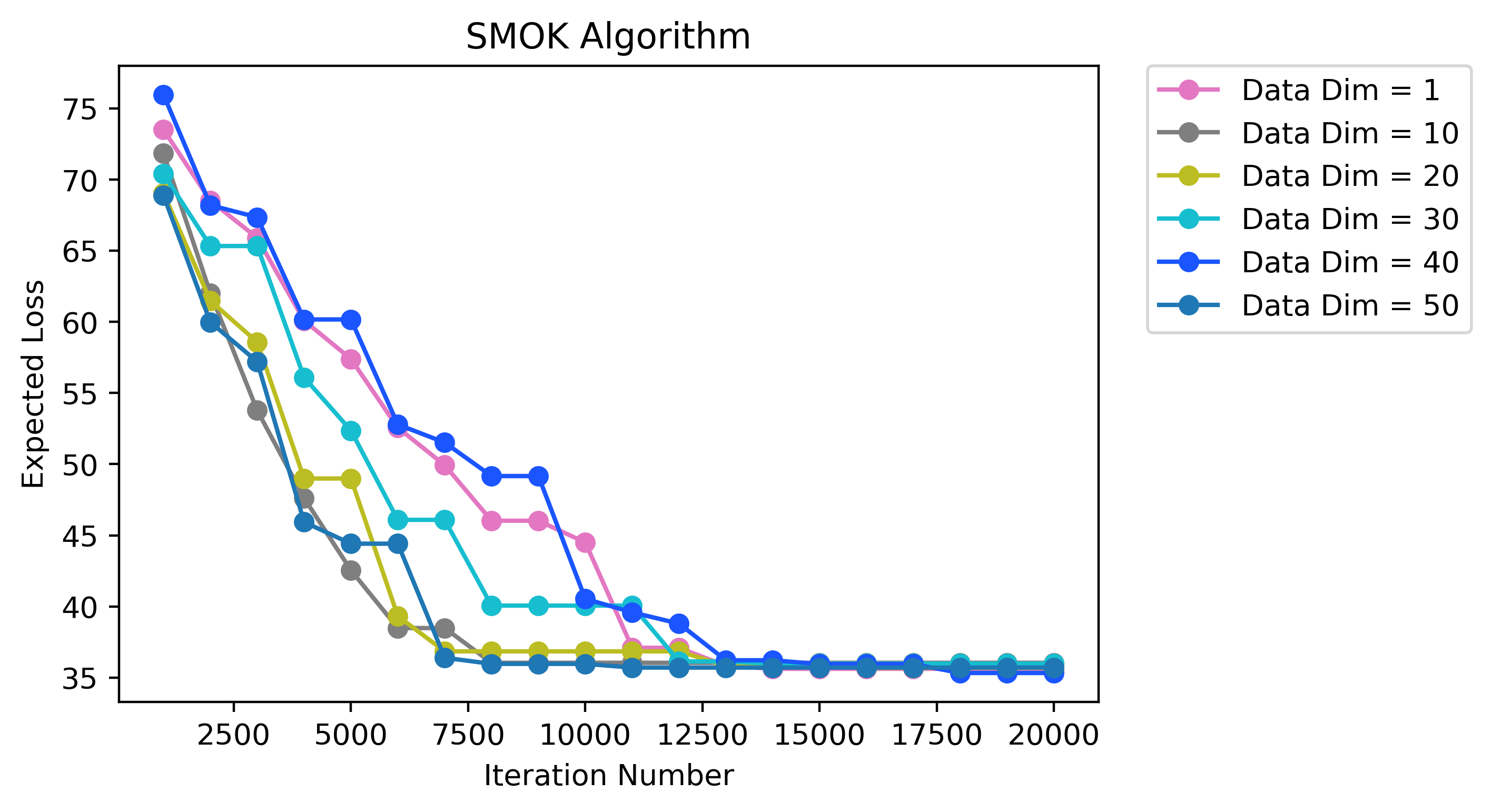}
         \caption{}
         \label{fig:data_dim_gd}
     \end{subfigure}
     \begin{subfigure}[b]{0.4\textwidth}
         \centering
         \includegraphics[trim=0 0 0 20,clip,width=\textwidth]{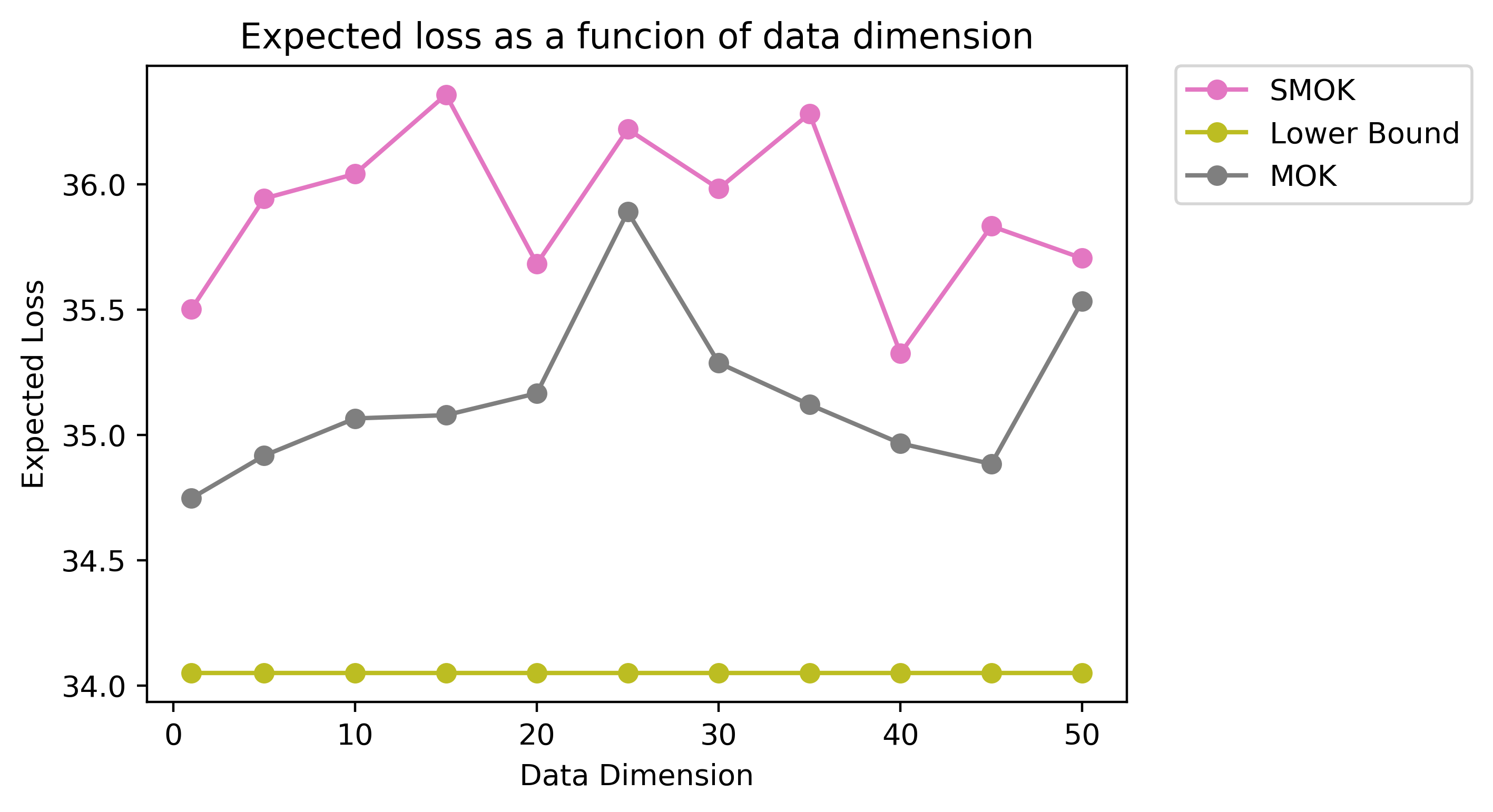}
         \caption{}
         \label{fig:data_dim_lowbound}
     \end{subfigure}
     \begin{subfigure}[b]{0.4\textwidth}
         \centering
         \includegraphics[width=\textwidth]{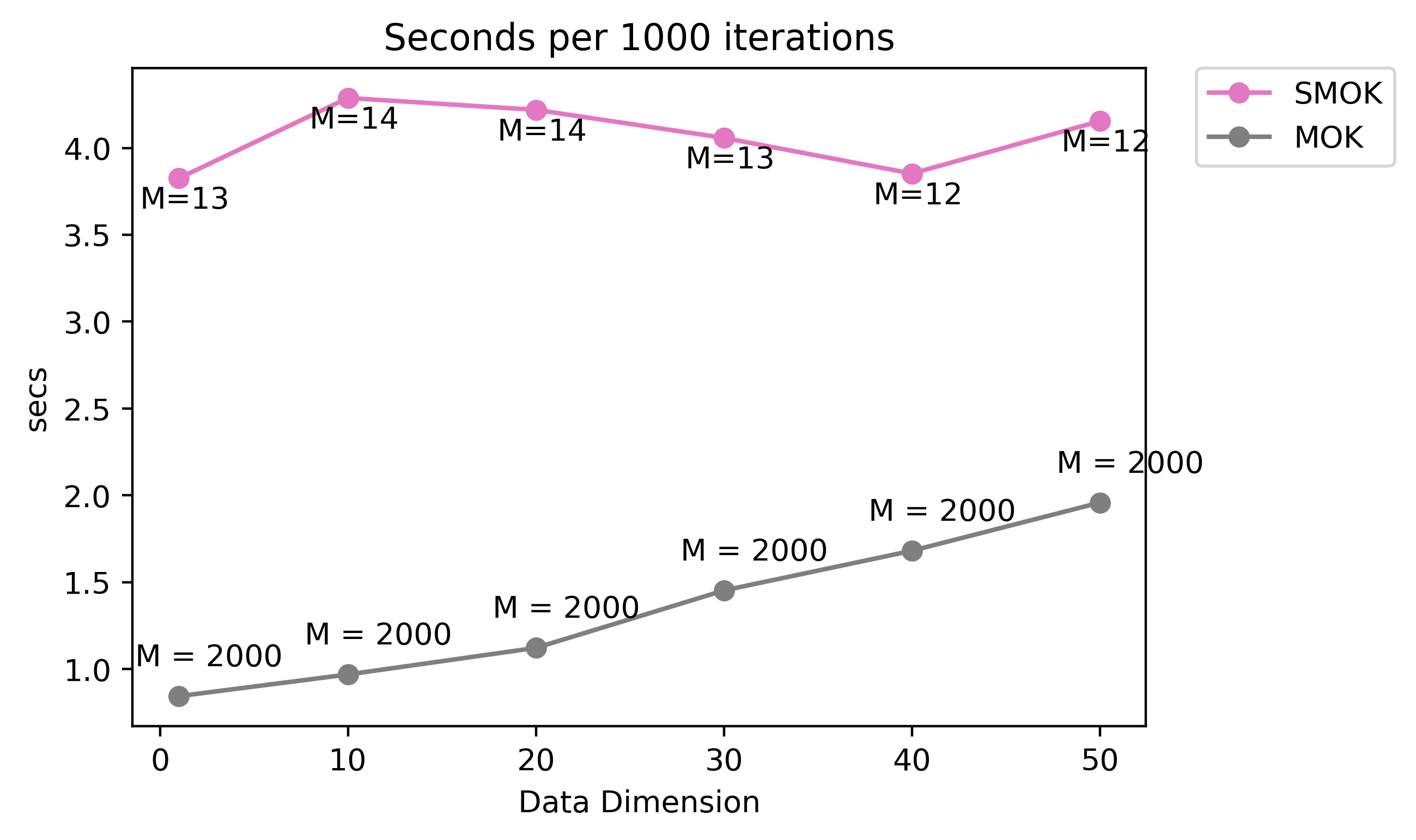}
         \caption{}
         \label{fig:data_dim_iter}
     \end{subfigure}
     \begin{subfigure}[b]{0.4\textwidth}
         \centering
         \includegraphics[width=\textwidth]{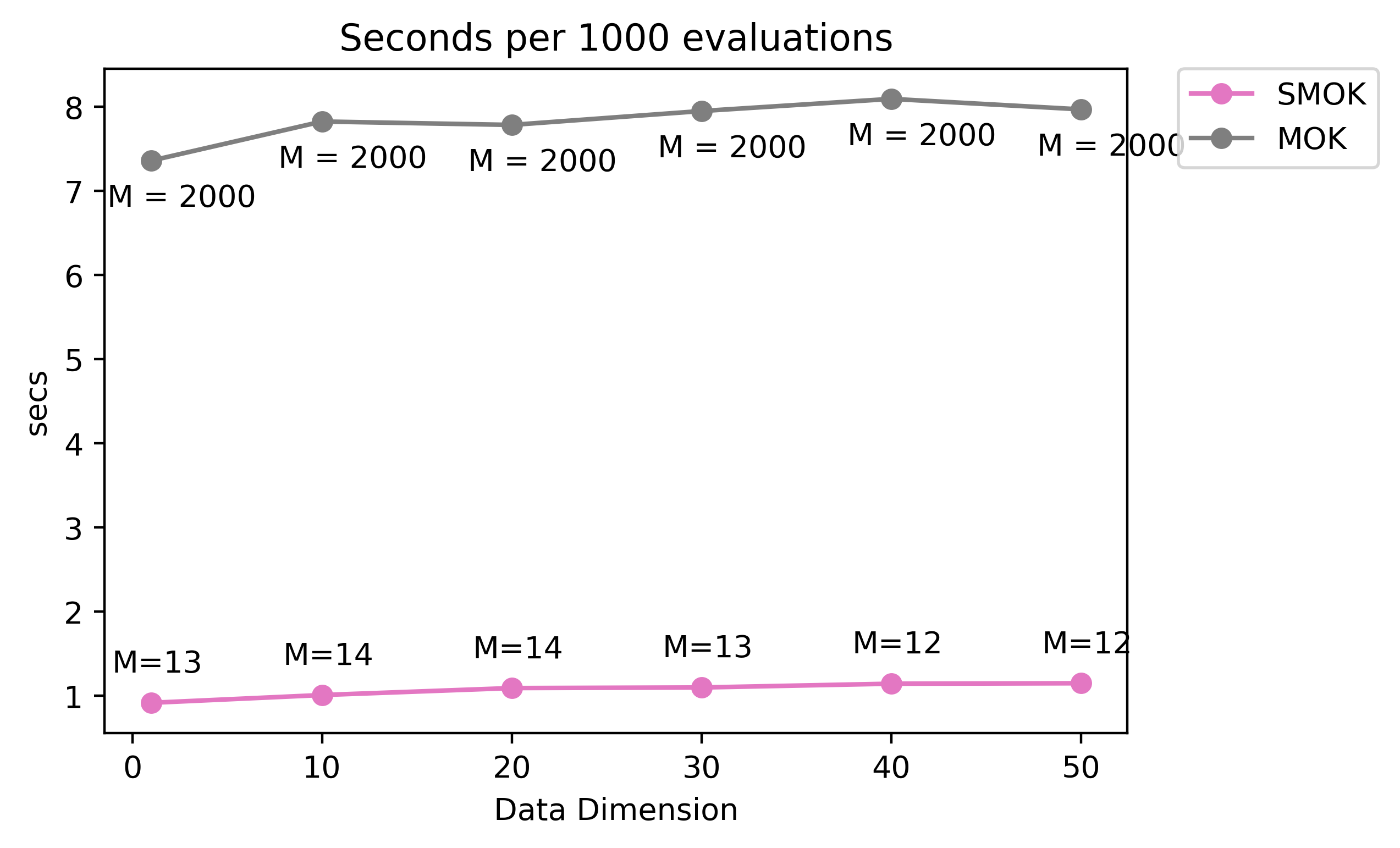}
         \caption{}
         \label{fig:data_dim_eval}
     \end{subfigure}
        \caption{Expected loss and computational time for varying data dimensions.}
        \label{fig:data_dim_loss}
\end{figure}

\subsubsection{Varying Control Dimension: ($L = 150, \ell =\frac{200}{r}, q = 1, T=3, N=2000, r = 1, 3, 5, 10$)}  {In order to make a fair comparison, we set $L = \frac{150}{r}$ and $\ell = \frac{200}{r}$, which guarantees that the optimal expected loss is the same across instances.} We observe in Figure \ref{fig:control_dim_lowbound} that the SMOK and MOK algorithms achieve very similar average out-of-sample loss across the different dimensions, and there are a couple of scenarios in which the pruning step helped to improve the expected loss. In addition, the number of iterations required for convergence (Figure \ref{fig:control_dim_gd}) does not seem to depend on the dimension of the control. Lastly, in Figure \ref{fig:control_dim_iter} we observe a slight linear increase in iteration time for both SMOK and MOK algorithms, with MOK having an advantage of around $4$ seconds per $1000$ iterations. In terms of evaluation time (Figure \ref{fig:control_dim_eval}) both algorithms grow linearly. As in the previous examples, the number of parameters for the SMOK algorithm is very low and varies between $13$ and $14$ across the different experiments.
\begin{figure}[!htp]
     \centering
     \begin{subfigure}[b]{0.4\textwidth}
         \centering
         \includegraphics[trim=0 0 0 20,clip,width=\textwidth]{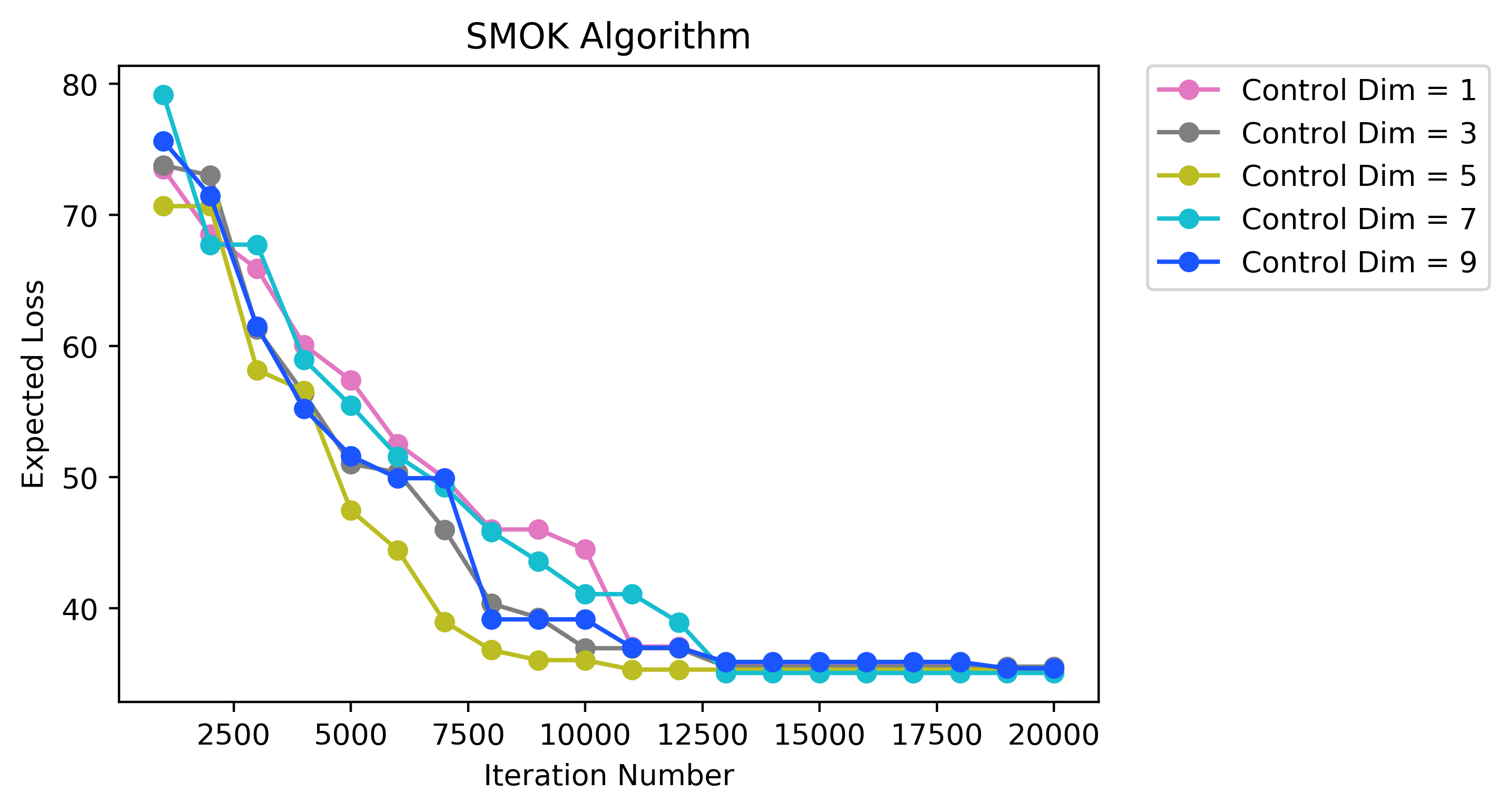}
         \caption{}
         \label{fig:control_dim_gd}
     \end{subfigure}
     \begin{subfigure}[b]{0.4\textwidth}
         \centering
         \includegraphics[trim=0 0 0 20,clip,width=\textwidth]{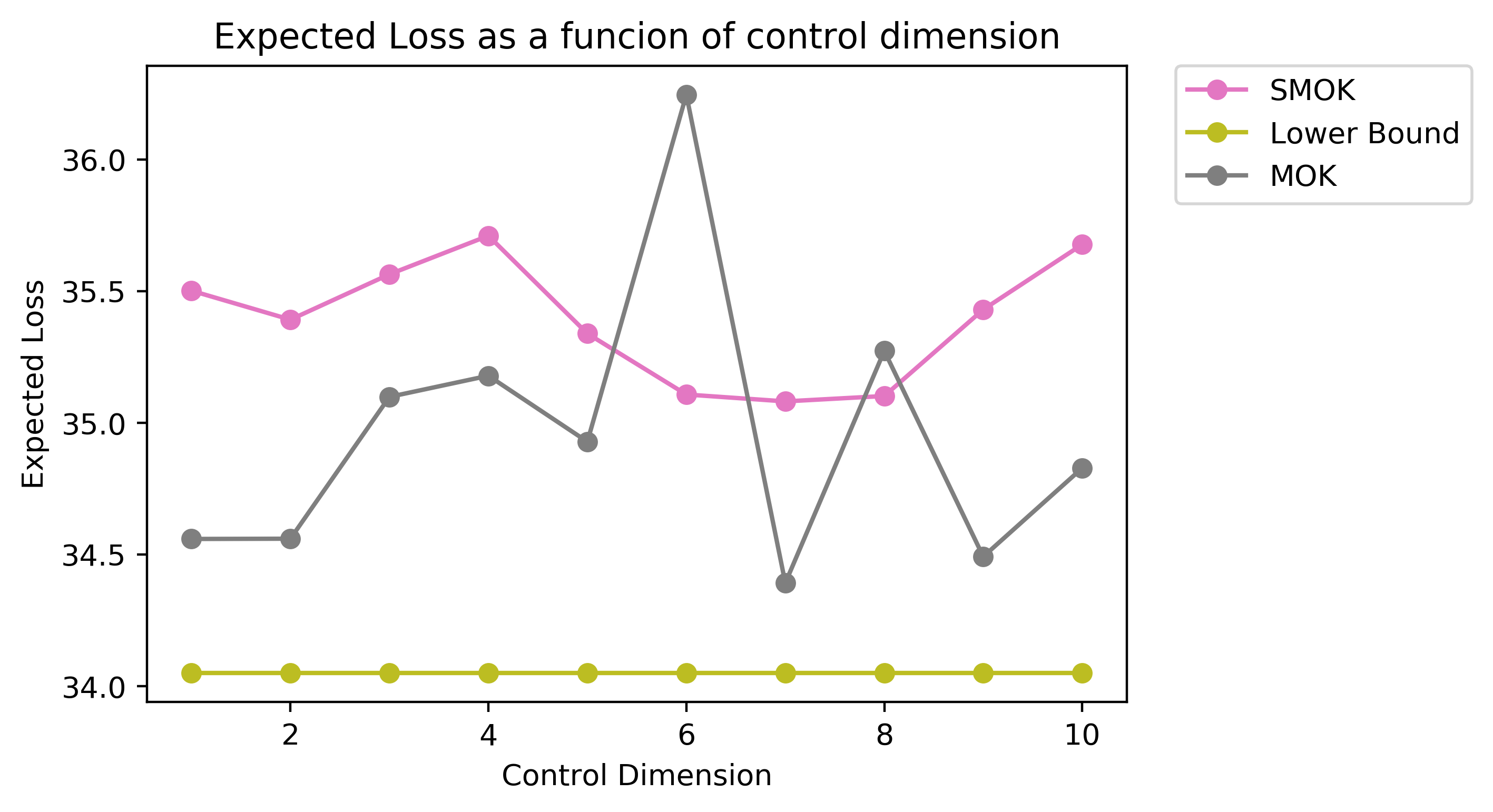}
         \caption{}
         \label{fig:control_dim_lowbound}
     \end{subfigure}
     \begin{subfigure}[b]{0.4\textwidth}
         \centering
         \includegraphics[width=\textwidth]{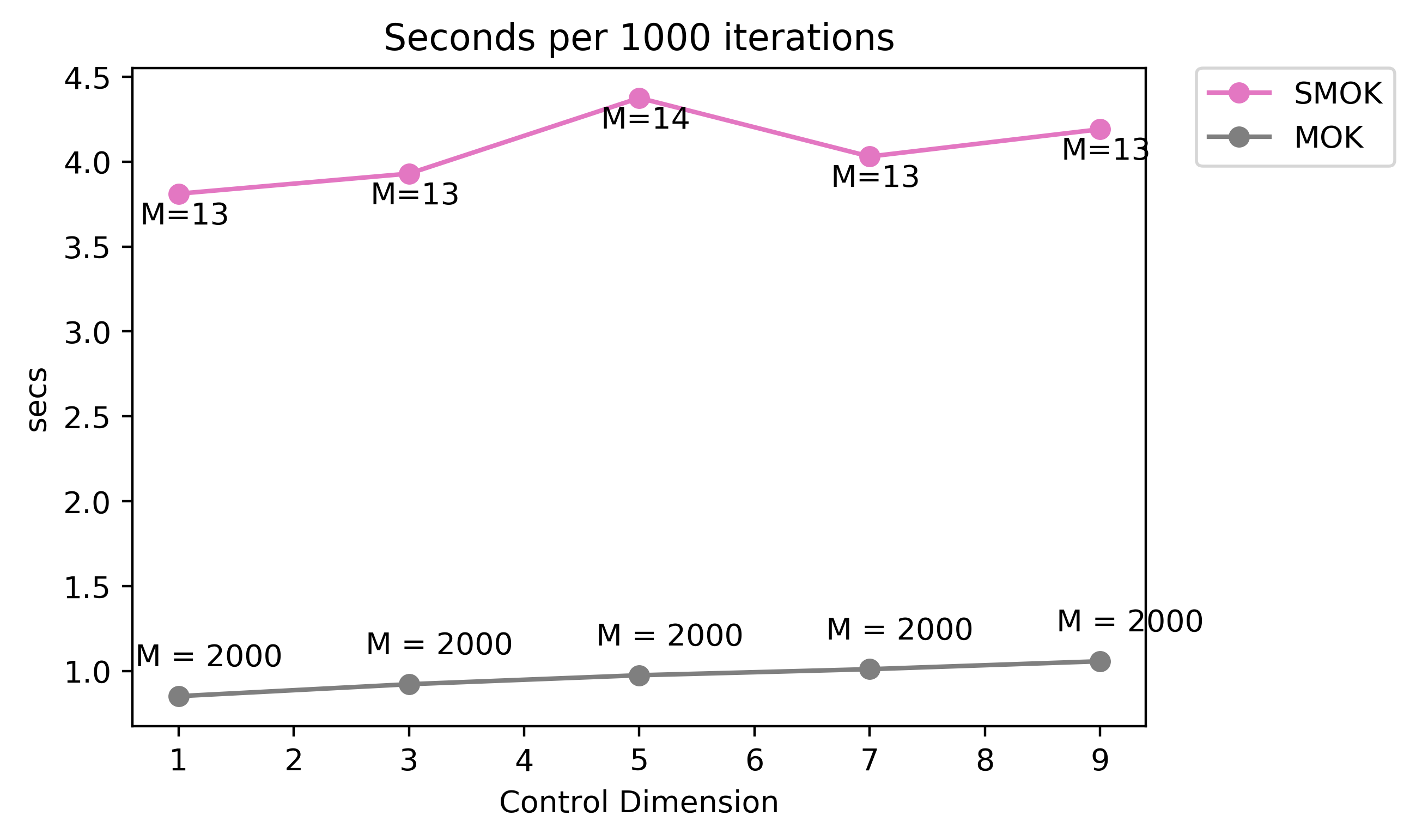}
         \caption{}
         \label{fig:control_dim_iter}
     \end{subfigure}
     \begin{subfigure}[b]{0.4\textwidth}
         \centering
         \includegraphics[width=\textwidth]{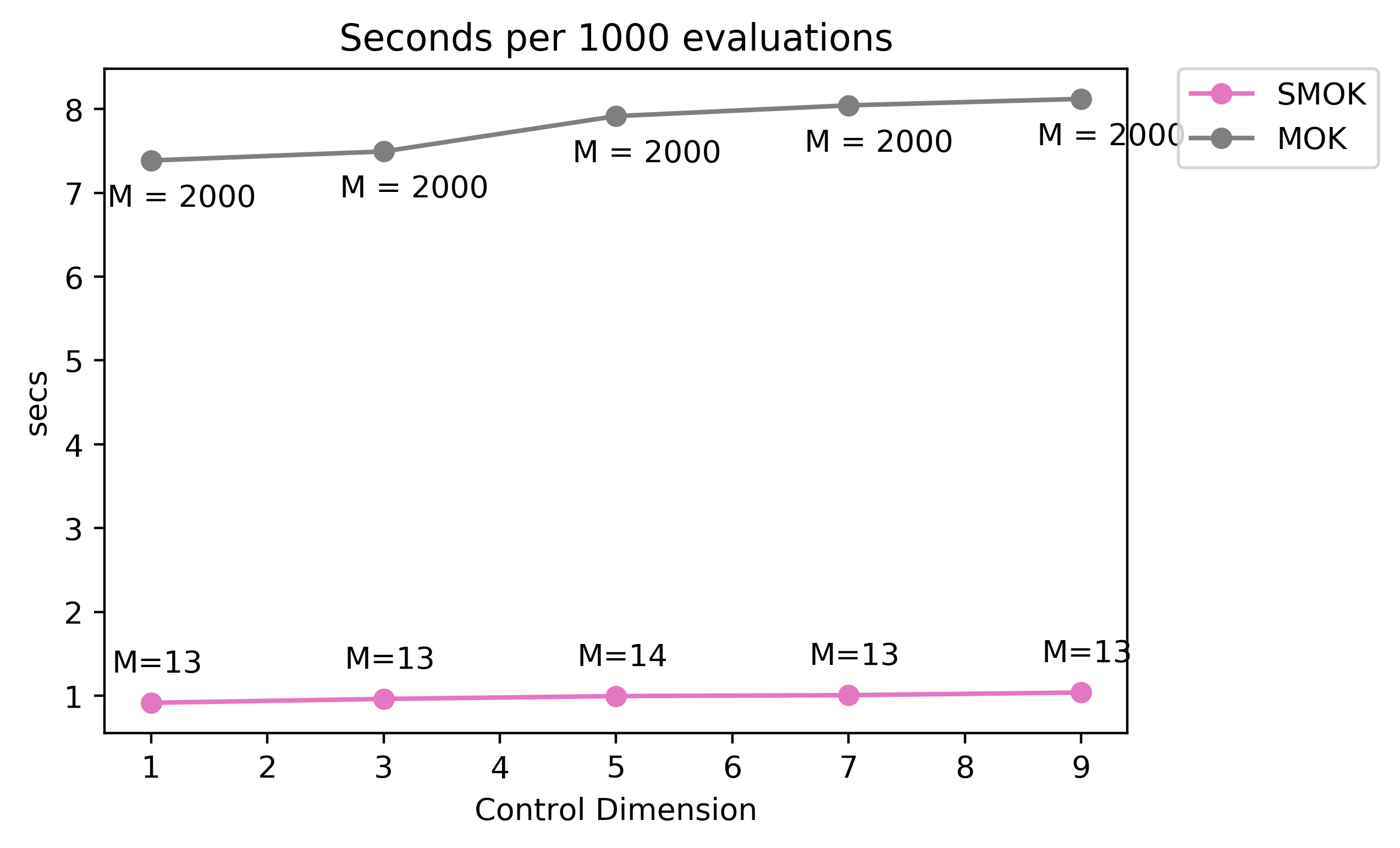}
         \caption{}
         \label{fig:control_dim_eval}
     \end{subfigure}
        \caption{Expected loss and computational time for varying control dimensions.}
        \label{fig:control_dim_loss}
\end{figure}

\subsection{Shipment Planning}
We next analyze a two-stage shipment planning problem, following the same problem setting as in \cite{bertsimas2022dynamic} and \cite{Bertsimas2015b}. In this example, a decision maker has access to side information {$\boldsymbol{x}$} (market trends, advertisements, etc.) and the goal is to ship items from the production facilities to multiple locations as to satisfy demand at minimum cost. First, the decision maker chooses an initial inventory quantity $u_{1f}\geq 0$ to be produced in each of the production facilities $f\in [F]$ at a per unit cost of $p_1$. Next, the demands $w_\ell \geq 0$ are observed in each location $\ell \in [L]$. If needed, the decision maker can produce additional units in each facility to satisfy demand, but at a higher per unit cost $p_2 >p_1$. Finally, demand is fulfilled by shipping $u_{2f\ell}$ units from facility $f$ to location $\ell$ at per-unit cost $c_{f\ell}$, and each unit of satisfied demand generates revenue $a>0$. The multistage optimization problem can then be written as
\begin{align*}
    \min_{\bu_1, \bu_2} & \mathbb{E}_{ \bw | \bx}\!\!\left[ \!p_1\!\!\sum_{f = 1}^F \!u_{1f} \!- a\!\!\!\sum_{\ell \in [L]} \!\!\!w_{\ell} + p_2\!\sum_{f = 1}^F\left[\sum_{\ell=1 }^L\! u_{2f\ell} - u_{1 f}\right]^+\! \!\!\!\!\!+\!\!\! \sum_{f = 1}^F \sum_{\ell = 1}^L \!c_{f\ell}\: u_{2f\ell} \bigg \vert  \bx = \bx_0 \!\right]\\
    &\text{s.t.} \quad \sum_{f = 1}^F u_{2f\ell} \geq w_\ell, \quad \forall \ell\in [L], \forall \boldsymbol{w} \in \mathcal{W},
\end{align*}
where $\mathcal{W}$ is the set of all possible demand realizations. We reproduced the computational experiments performed in \cite{bertsimas2022dynamic} using the same parameters, the same data generation procedure as well as the same data set sizes. More specifically, we use $F = 4, L=12, p_1 = 5, p_2 = 100$ and $a = 90$.  The costs $\boldsymbol{c}$ and covariates $\boldsymbol{x}$ are also generated in an identical manner as in \cite{bertsimas2022dynamic}.

We compare the SMOK and MOK algorithms against  \textbf{SRO}, \textbf{SRO-knn} and \textbf{SAA-knn}. We train all methods over 100 independent training sets and evaluate them on a test set of size 100. The average out-of-sample profits achieved across the different methods are shown in Table \ref{table:shipment-results}. We observe that both MOK and SMOK outperform the other methods, with MOK achieving the highest revenues. However, as observed in Table \ref{table:shipment-time}, only the SAA and SMOK methods have tractable growth as the data size increases. In particular, the SMOK algorithm achieves high accuracies using only around $60$ parameters per decision even when the data size increases to large numbers.

\begin{table}[!htp]
    \begin{tabular}{|P{0.6cm}|P{1cm}|P{1cm}|P{1.2cm}|P{1.2cm}|P{1.2cm}|P{1.2cm}|P{1cm}|}
    \hline
         \scalebox{0.9}{$\boldsymbol{N}$} &  \scalebox{0.9}{\textbf{SRO}} &  \scalebox{0.9}{\textbf{SRO}} & \scalebox{0.9}{\textbf{SRO-knn}}  &  \scalebox{0.9}{\textbf{SRO-knn}}   & \scalebox{0.9}{\textbf{SAA-knn}} & \scalebox{0.9}{\textbf{MOK}} & \scalebox{0.9}{\textbf{SMOK}}  \\
          & \scalebox{0.9}{($\epsilon = 100$)} & \scalebox{0.9}{($\epsilon = 500$)} & \scalebox{0.9}{($\epsilon = 100$)}& \scalebox{0.9}{($\epsilon = 500$)}& & &\\
         \hline
         \hline
         \boldsymbol{$100$}& 160007.0 & 159866.7 & 157522.9 & 158671.5 & 156639.6 & \textbf{161536.9}&160737.0\\
         \hline
         \boldsymbol{$200$}& 160221.1 & 160075.0 & 157863.5 &159136.9 & 156911.9&\textbf{164050.1}& 163039.2\\
         \hline
         \boldsymbol{$300$} & 160431.0 & 160145.6 & 158697.6 & 159656.2& 157669.6& \textbf{164860.6}&163703.8\\
         \hline
    \end{tabular}
    \caption{Out-of-sample profit for the shipment planning problem.}
    \label{table:shipment-results}
\end{table}

\begin{table}[!htp]
    \centering
    \begin{tabular}{|P{0.6cm}|P{1cm}|P{1cm}|P{1.2cm}|P{1.2cm}|P{1.2cm}|P{1.2cm}|P{1cm}|}
    \hline
         \scalebox{0.9}{$\boldsymbol{N}$} & \scalebox{0.9}{\textbf{SRO}} & \scalebox{0.9}{\textbf{SRO}} &\scalebox{0.9}{\textbf{SRO-knn}} &  \scalebox{0.9}{\textbf{SRO-knn}}   & \scalebox{0.9}{\textbf{SAA-knn}} & \scalebox{0.9}{\textbf{MOK}} & \scalebox{0.9}{\textbf{SMOK}}  \\
          &\scalebox{0.9}{($\epsilon = 100$)} & \scalebox{0.9}{($\epsilon = 500$)} & \scalebox{0.9}{($\epsilon = 100$)} & \scalebox{0.9}{($\epsilon = 500$)} & & & \\   
         \hline
         \hline
         \boldsymbol{$100$}& 8 & 6 & 30 & 35 & 4 & {38}&150\\
         \hline
         \boldsymbol{$200$}& 11 & 12 & 78 & 75 &  4&{ 42}& 260\\
         \hline
         \boldsymbol{$300$} &19 & 21 & 125 &  132 &  4& {46}&240\\
         \hline
         \boldsymbol{$500$} &38 & 39 & 276 & 280 &  5& 50& 245\\
         \hline
         \boldsymbol{$1000$} & 74 & 76 & 772 &  790&  10& {65}& 255\\
         \hline
         \boldsymbol{$5000$} & 559 & 581 & 54000 & 54100& 54& {288}& 252\\
         \hline
    \end{tabular}
    \caption{Total computation time (seconds) for solving one instance of the shipment planning problem.}
    \label{table:shipment-time}
\end{table}

\section{Conclusion}
In this work, we developed a tractable data-driven approach for solving multistage stochastic optimization problems in which the uncertainties are independent of previous decisions. We represented the decision rules as elements of a reproducing kernel Hilbert space and performed functional stochastic gradient descent to minimize the empirical regularized loss. We next incorporated sparsification techniques based on function subspace projections, which decreased the number of parameters per controller. We prove that the proposed approach is asymptotically optimal for multistage stochastic programming with side information. 

The practical value of the proposed data-driven approach was shown across various computational experiments on stochastic inventory management problems, demonstrating that it produces high-quality decisions, does not worsen in multidimensional settings and remains tractable even with large data sizes. This approach does not rely on the traditional use of approximation with scenario trees, and provides a novel method for leveraging advances in machine learning to solve multistage stochastic optimization problems.\\

\noindent \textbf{Statements and Declarations.} The authors declare that no funds, grants, or other support were received during the preparation of this manuscript. The authors have no relevant financial or non-financial interests to disclose.
\newpage

%\begin{acknowledgements}
%\end{acknowledgements}


\bibliographystyle{spbasic}      % basic style, author-year citations

\bibliography{references}
\newpage
\begin{appendix}

\section{Reproducing Kernel Hilbert Spaces Overview} \label{sec:rkhs}
A reproducing kernel Hilbert space (RKHS) is a Hilbert space in which the elements are functions that preserve pointwise distance. Specifically, if two functions are close with respect to the Hilbert space norm, then their pointwise evaluations are close with respect to the norm of the functions' output space. Each RKHS is generated by a positive definite kernel $K(\cdot, \cdot)$; a function $K:~\mathcal{Z} \times \mathcal{Z}  \to \mathbb{R}$ satisfying
\[\sum_{i=1}^m \sum_{j=1}^m a_i a_j K(\bz^i, \bz^j) \geq 0, \quad \forall m \in \mathbb{N}, \quad \bz^1,\ldots  ,\bz^m \in \mathcal{Z}, \quad  a_1,\ldots,a_m \in \mathbb{R} \,. \]

\begin{definition}\label{rkhs}
	A reproducing kernel Hilbert space $\mathcal{H}$ generated by a positive definite kernel $K: \mathcal{Z}  \times \mathcal{Z}  \to \mathbb{R}$ is the closure of the set of functions
	\[\Set{f: \mathcal{Z}  \to \mathbb{R} | \: f(\bz) = \sum_{c=1}^{L} a^c K(\bz^c, \bz), \; \text{for } \bz^1,\ldots,\bz^L \in \mathcal{Z}  \text{ and } L \in \mathbb{N} }, \]
	with inner product of $f_1(\bz) = \sum_{c =1}^{L_1} a_1^c K(\bz_1^c,\bz)$ and $f_2(\bz) = \sum_{c = 1}^{L_2} a_2^c K(\bz_2^c,\bz)$ defined as 
	$$ \langle f_1,f_2\rangle_{\mathcal{H}}= \sum_{c_1=1}^{L_1} \sum_{c_2 =1}^{L_2} a_1^{c_1} a_2^{c_2} K(\bz_1^{c_1}, \bz_2^{c_2}).$$
\end{definition}
The complexity of a reproducing kernel Hilbert space  depends on the kernel generating it. A linear kernel, for example, generates the Hilbert space of linear functions. A Gaussian kernel generates much more complex spaces; it has the property that for {compact spaces} $\mathcal{Z}$ it generates spaces that are dense in $C(\mathcal{Z})$ (the space of continuous bounded funcions on $\mathcal{Z}$) {in the maximum norm}. Kernels with this property are called \textit{universal kernels} {\citep{micchelli2006universal}} and they are very useful for solving functional optimization problems over continuous functions, since the problem can be solved over the RKHS instead.

One of the main reasons why RKHSs are so useful when working with data is the fact that they transform pointwise evaluation into an inner product of elements in the Hilbert space, and vice-versa. Specifically, if $f$ belongs to  the reproducing kernel Hilbert space $\mathcal{H}$ generated by kernel $K$, we have  
\begin{align}f(\bx) = \langle K(\bx, \cdot), f \rangle_{\mathcal{H}}\label{rep_prop},\end{align} %DB the notation _H was not defined. Please clean up the notation. %KIM I have added the _H subscript in the definition of inner product above.
for all $\bx$ in the domain of $f$. This equivalence is known as the \textit{reproducing property}. The next result, known as the Representer Theorem, illustrates how in many cases solving functional optimization problems over a RKHS is equivalent to solving an optimization problem over a real space, and the proof relies mostly on the reproducing property.

\begin{proposition}[Representer Theorem]\label{repthmsingle}
	 Suppose we have a data matrix $\bbz = [\bz^1,\ldots, \bz^N]$ for some fixed data points $\bz^1, \hdots , \bz^N\in \mathcal{Z}$. Let $\mathcal{H}$ be the reproducing kernel Hilbert space generated by a kernel $K:  \mathcal{Z} \times  \mathcal{Z} \to \mathbb{R}$. Then, for any arbitrary loss function $c: \mathbb{R} \times  \mathcal{Z} \to \mathbb{R}$ and any regularization parameter $\lambda \geq 0$, there exists a solution to
	\begin{equation}\label{repsingleoriginaleq}
	\inf_{h \in \mathcal{H}} \frac{1}{N} \sum_{n=1}^N c(h(\bz^n), \bz^n) + \frac{\lambda}{2} \|h\|_{\mathcal{H}}^2
	\end{equation}
	that takes the form
	\begin{equation}\label{repsingleform}
	h^*(\cdot ) = \sum_{n = 1}^N a_n K(\bz^n, \cdot) \,,
	\end{equation}
	for some scalars $a_1,\ldots, a_N \in \mathbb{R}^N$. 
\end{proposition}
The Representer Theorem implies that the solution to the functional optimization problem \eqref{repsingleoriginaleq} can be found by solving instead the following finite dimensional optimization problem
	\begin{equation}\label{repsinglefinaleq}
	\min_{\ba \in \mathbb{R}^N} \frac{1}{N}\sum_{n=1}^N c\big((\bbk[\bbz, \bbz]\ba)_n, \bz^n\big) + \frac{\lambda}{2}  \ba^T  \bbk[\bbz, \bbz] \ba \,,
	\end{equation}
where $\bbk[\bbz, \tilde{\bbz}]$ is the kernel matrix (between equal size matrices $\bbz, \tilde{\bbz}$) whose $(n,m)$ component is $K(\bz^n, \tilde{\bz}^m)$.

The proof of this theorem follows from the fact that any function in $\mathcal{H}$ can be decomposed as the sum of a function of the form \eqref{repsingleform} and a function orthogonal to every function of this form.  The theorem then follows after showing that, thanks to the reproducing property, the sum in the objective of \eqref{repsingleoriginaleq} is independent of the orthogonal part, and the second term in the objective is increasing in the orthogonal part. This theorem can be extended to a multidimensional version in which the optimization problem is over multiple functions $h_1,\ldots, h_r\in \mathcal{H}$. In this case, the Representer Theorem tells us that there exists a solution $\bh^*\in \mathcal{H}^r$ that takes the form
\begin{equation}\label{repsingleform2}
	\bh^*(\cdot ) = \bba \boldsymbol{K}(\bbz, \cdot)  \,,
	\end{equation}
where $\boldsymbol{K}(\bbz, \cdot) \coloneqq [K(\bz^1, \cdot),\ldots, K(\bz^N, \cdot)]^T$ and $\bba\in \mathbb{R}^{r\times N}$. For more details about the representer theorem's proof and applications we refer the reader to { \cite{wahba1990spline, soentpiet1999advances, scholkopf2002learning,  shawe2004kernel}}.

\section{Lemmas}
In this appendix we will state and proof several lemmas needed for the proof of Theorems $\ref{convergence-rates}$ and Corollary $\ref{corollary}$. For generality, we consider the constrained problem in Eq. \ref{constrained_version}, and define:
\begin{align}
c^\psi\big(\bu(\bz), \bz\big) &\coloneqq  c\big(\bu(\bz), \bz\big) + \psi\sum_{q = 1}^Q\max\Big(0, g_q\big(\bu(\bz)\big)\Big)^2, \label{augmented_loss2} \\ 
E(\bu) & \coloneqq \mathbb{E}_{\bz}\left[c(\bu(\bz), \bz)\right], \label{expected_cost2}\\
E^\psi(\bu) & \coloneqq \Exp_\bz \left[ c^\psi\big(\bu(\bz), \bz\big) \right], \label{expected_loss2}\\
E^{\lambda, \psi}(\bu)& \coloneqq E^\psi(\bu)  + \frac{\lambda}{2} \|\bu\|_{\mathcal{H}}^2, \label{exp_reg_loss2}\\
E^{\lambda, \psi}_\mathcal{S}(\bu)&  \coloneqq \frac{1}{N}\sum_{n=1}^N c^\psi\big(\bu(\bz^n), \bz^n\big)  + \frac{\lambda}{2} \|\bu\|_{\mathcal{H}}^2, \label{emp_reg_loss2}\\
E^{\lambda, \psi}_n(\bu)&  \coloneqq  c^\psi\big(\bu(\bz^n),\bz^n\big)  + \frac{\lambda}{2} \|\bu\|_{\mathcal{H}}^2 \label{stoc_reg_loss2}.
\end{align}

\begin{lemma}
\label{2norm}
Under assumption 1, we have  
\[\|\bu(\bz)\|_2\leq \kappa\|\bu\|_{\mathcal{H}} \quad \forall \: \: \bz\in \mathcal{Z}, \: \bu\in \mathcal{H}.\]
\end{lemma}

\begin{proof}
Let $\bu\in \mathcal{H}$ and $\bz\in \mathcal{Z}$. We have
    \begin{align*}
    \|\bu(\bz)\|^2_2 &=\sum_{t = 1}^T\sum_{i = 1}^{r_t} u_{t,i}(\bz_{1:t-1})^2 = \sum_{t = 1}^T\sum_{i = 1}^{r_t} \langle u_{t, i}, K_t(\bz_{1:t-1}, \cdot)\rangle_{\mathcal{H}_{t}}^2&  &\text{(by Eq. $(\ref{rep_prop})$),}\\ &\leq  \sum_{t = 1}^T\sum_{i = 1}^{r_t} \|K_t(\bz_{1:t-1}, \cdot)\|^2_{\mathcal{H}_{t, i}}\|u_{t, i}\|^2_{\mathcal{H}_{t}}&  &\text{(Cauchy-Schwarts Ineq.),}\\
    &\leq \kappa^2 \|\bu\|^2_{\mathcal{H}}&  &\text{(Assumption $\ref{assumption_kernel}$)},
    \end{align*}
and the lemma follows. \hfill
\end{proof}

\begin{lemma}
\label{Hnorm}
Under Assumptions 1 and 2, we have 
\begin{itemize}
    \item [a)] \label{lemma2a} The true minimizer $\bu^{\lambda, \psi}$ of $E^{\lambda, \psi}$ defined in \ref{exp_reg_loss2}, satisfies
    \[
    \|\bu^{\lambda, \psi}\|_{\mathcal{H}}\leq \frac{\kappa C}{\lambda}.
    \]
    \item [b)] \label{lemma2b}If Algorithm \ref{SMOK} is initialized such that $\|\bu^0\|_{\mathcal{H}}\leq \frac{\kappa C}{\lambda}$, then
    \[
    \hspace{2.2cm}  \|\bu^n\|_{\mathcal{H}}\leq \frac{\kappa C}{\lambda} \quad \forall \: \: n\in [N].
    \]
\end{itemize}
\end{lemma}
\begin{proof}
\textit{a)} We proceed as in \cite{kivinen2004online}. We define $\bu^{\lambda, \psi}_{\mathcal{S}}$ as the minimizer of $E^{\lambda, \psi}_\mathcal{S}$, and $\hat{\bu} = (1-\epsilon)\bu^{\lambda, \psi}_\mathcal{S}$ for $\epsilon>0$. We have
\begin{align*}
 0 &\leq  E^{\lambda, \psi}_\mathcal{S}(\hat{\bu}) - E^{\lambda, \psi}_{\mathcal{S}}({\bu^{\lambda, \psi}_\mathcal{S}})& &\text{(by Optimality of } \bu^{\lambda, \psi}_\mathcal{S}),\\
 &= \frac{1}{N}\sum_{n = 1}^N \Big( c^{\psi}(\hat{\bu}(\bz^n), \bz^n) - c^{\psi}({\bu}^{\lambda, \psi}_\mathcal{S}(\bz^n), \bz^n)\Big)  + \frac{\lambda}{2}  \big(\|\hat{\bu}\|_{\mathcal{H}}^2 - & &\|{\bu}^{\lambda, \psi}_\mathcal{S}\|_{\mathcal{H}}^2\big)\\
 & \leq \frac{C}{N} \sum_{n = 1}^N \|\hat{\bu}(\bz^n) - \bu^{\lambda, \psi}_\mathcal{S}(\bz^n)\|_{2} + \frac{\lambda}{2} \big((1-\epsilon)^2-1\big) \|\bu^{\lambda, \psi}_\mathcal{S}\|_{\mathcal{H}}^2& &\text{(by Assumption } \ref{assumption_lip}), \\
 &\leq \frac{\kappa C}{N} \sum_{n = 1}^N \|\hat{\bu} - \bu^{\lambda, \psi}_\mathcal{S}\|_{\mathcal{H}} -\lambda\epsilon\|\bu^{\lambda, \psi}_\mathcal{S}\|_{\mathcal{H}}^2 +  \frac{\lambda}{2}\epsilon^2 \|\bu^{\lambda, \psi}_\mathcal{S}\|_{\mathcal{H}}^2& &\text{(by Lemma } \ref{2norm}),\\
 &= \kappa C \epsilon\|\bu^{\lambda, \psi}_\mathcal{S}\|_{\mathcal{H}} -\lambda\epsilon\|\bu^{\lambda, \psi}_\mathcal{S}\|_{\mathcal{H}}^2 +  \frac{\lambda}{2}\epsilon^2 \|\bu^{\lambda, \psi}_\mathcal{S}\|_{\mathcal{H}}^2.& &
\end{align*}
Dividing by $\epsilon\|\bu^{\lambda, \psi}_\mathcal{S}\|_{\mathcal{H}}$ on both sides and taking the limit as $\epsilon\rightarrow 0$ we obtain $\|\bu^{\lambda, \psi}_\mathcal{S}\|_{\mathcal{H}}\leq \frac{\kappa C}{\lambda}$ and the desired result then follows by taking the limit as $N\rightarrow \infty$. \\

\noindent \textit{b)} To prove the upper bound for the decisions output by the algorithm in each iteration we proceed by induction on the iteration number $n$. We have $\|\bu^0\|_{\mathcal{H}}\leq \frac{\kappa C}{\lambda}$ by assumption. Suppose the bound holds for $n$. Then, we have
\begin{align*}
    \|\bu^{n+1}\|_{\mathcal{H}} &= \big\| \Pi_{\bbd^{n+1}}[ (1 - \eta_n\lambda)\bu^{n} - \eta_n \nabla_{\bu}c^\psi(\bu^n(\bz^n), \bz^n)]\big\|_{\mathcal{H}}&  &\text{(by definition } \ref{def_proj}),\\
     &\leq \| (1 - \eta_n\lambda)\bu^{n} - \eta_n \nabla_{\bu}c^\psi(\bu^n(\bz^n), \bz^n)\|_{\mathcal{H}}& &\textup{(by definition of $\Pi_{\bbd^{n+1}})$}\\
     &\leq (1 - \eta_n\lambda)\| \bu^{n} \|_{\mathcal{H}} + \eta_n\|\nabla_{\bu}c^\psi(\bu^n(\bz^n), \bz^n)\|_{\mathcal{H}}&  &\textup{(by triangle inequality),}\\
     &\leq (1 - \eta_n\lambda)\| \bu^{n} \|_{\mathcal{H}} + \eta_n\kappa\|\nabla_{\bu(\bz)}c^\psi(\bu^n(\bz^n), \bz^n)\|_2&  &\textup{(by Eq. \eqref{grad_computation}),}\\
     &\leq (1 - \eta_n\lambda)\frac{\kappa C}{\lambda} + \eta_n\kappa\|\nabla_{\bu(\bz)}c^\psi(\bu^n(\bz^n), \bz^n)\|_2&  &\textup{(by assumption for $n$),}\\
     & \leq (1 - \eta_n\lambda)\frac{\kappa C}{\lambda} + \eta_n\kappa C = \frac{\kappa C}{\lambda}& &\textup{(by  Eq. \ref{gradient bound}),}
\end{align*}
and therefore the result holds for all $n\in \mathbb{N}$ as desired. \hfill
\end{proof}
\begin{lemma}
\label{lemma_grad_bound}
    Under Assumptions 1-3, for any $\bu\in \mathcal{H}$ satisfying $\|\bu\|_{\mathcal{H}}\leq \frac{\kappa C}{\lambda}$, we have
    \[
    \Exp\left[\|\nabla_{\bu}E^{\lambda, \psi}_n(\bu)\|^2_{\mathcal{H}_t}\right] \leq 4\kappa^2C^2.
    \]
\end{lemma}
\begin{proof}
    Fix some $t\in \{1,\ldots, T\}$. Using the fact that $\|\ba + \bb\|^2_{\mathcal{H}}\leq 2(\|\ba\|^2_{\mathcal{H}} + \|\bb\|^2_{\mathcal{H}})$ for any $\ba, \bb \in \mathcal{H}$, as well as Assumption \ref{assumption_convex},  we obtain that for any $\bu\in \mathcal{H}$, it holds
    \begin{align*}
        \Exp\left[\|\nabla_{\bu}E^{\lambda, \psi}_n(\bu)\|^2_{\mathcal{H}}\right] &\leq 2\Exp\left[\|\nabla_{\bu}c^\psi(\bu(\bz^n), \bz^n)\|^2_{\mathcal{H}}\right] + 2{\lambda^2}\|\bu\|^2_{\mathcal{H}},& &\\
        &\leq 2\kappa^2\Exp\left[\|\nabla_{\bu(\bz)}c^\psi(\bu(\bz^n), \bz^n)\|^2_{\mathcal{H}}\right] + 2{\lambda^2}\|\bu\|^2_{\mathcal{H}}&  &\textup{(By Eq. \eqref{grad_computation}),}\\
        &\leq 2\kappa^2\Exp\left[\|\nabla_{\bu(\bz)}c^\psi(\bu(\bz^n), \bz^n)\|^2_{\mathcal{H}}\right] + 2{\lambda^2}\left(\frac{\kappa C}{\lambda}\right)^2&  &\textup{(By Lemma \ref{Hnorm}),}\\
        &\leq 2\kappa^2C^2 + 2{\lambda^2}\left(\frac{\kappa C}{\lambda}\right)^2&  &\textup{(by  Eq. \eqref{gradient bound}),}\\
        &= 4\kappa^2C^2.
    \end{align*}
    \hfill
\end{proof}
\begin{lemma}
\label{lemma4}
 Under Assumption 3, given independently and identically distributed realizations $\{\bz^n\}$ of $\bz$, we have
\[
\|\nabla_{\bu}E^{\lambda, \psi}_n(\bu^n) - \tilde{\nabla}_{\bu}E^{\lambda, \psi}_n(\bu^n)\|_{\mathcal{H}}\leq \frac{\epsilon_n}{\eta_n},
\]
where $\tilde{\nabla}_{\bu}E^{\lambda, \psi}_n(\bu^n)$ was defined in Eq. \eqref{proj_grad}.
\end{lemma}
\begin{proof}
By definition of $\tilde{\nabla}_{\bu}E^{\lambda, \psi}_n(\bu^n)$ we have that 
\[
\begin{split}
&\|\nabla_{\bu}E^{\lambda, \psi}_n(\bu^n) - \tilde{\nabla}_{\bu}E^{\lambda, \psi}_n(\bu^n)\|^2_{\mathcal{H}}\\ =& \bigg\|\nabla_{\bu}E^{\lambda, \psi}_n(\bu^n) - \frac{\bu^n - \Pi_{\bbd^{n+1}}[\bu^n - \eta_n\nabla_{\bu}E^{\lambda, \psi}_n(\bu^n)]}{\eta_n}\bigg\|^2_{\mathcal{H}},\\
=& \Big\| \frac{1}{\eta_n}\Pi_{\bbd^{n+1}}[\bu^n - \eta_n\nabla_{\bu}E^{\lambda, \psi}_n(\bu^n)] - \frac{1}{\eta_n}(\bu^n - \eta_n\nabla_{\bu}E^{\lambda, \psi}_n(\bu^n)) \Big\|^2_{\mathcal{H}},\\ 
=& \frac{1}{\eta_n^2}\big\|\bu^{n+1} - \tilde{\bu}^{n+1} \big\|^2_{\mathcal{H}},
\end{split}
\]
where $\tilde{\bu}^{n+1} \coloneqq \bu^n - \eta_n\nabla_{\bu}E^{\lambda, \psi}_n(\bu^n)$ is the result of applying one FSGD iteration to $\bu^n$. By the stopping criterion of the KOMP algorithm we know that $\big\|\bu^{n+1} - \tilde{\bu}^{n+1} \big\|_{\mathcal{H}} \leq \epsilon_n$, and therefore the lemma follows after taking the square root on both sides. \hfill
\end{proof}
\begin{lemma}
\label{lemma3}
Under Assumption 3, for any $\bu\in \mathcal{H}$ with $\|\bu\|_{\mathcal{H}}\leq \frac{\kappa C}{\lambda}$, we have
\begin{align}
\label{eq_lemma5}
&E^{\lambda, \psi}_n(\bu^n) - E^{\lambda, \psi}_n(\bu)\\
\leq&  \frac{1}{2\eta_n}\left(\|\bu^n - \bu\|^2_{\mathcal{H}} - \|\bu^{n+1} - \bu\|^2_{\mathcal{H}} \right) + \eta_n\|\nabla_{\bu}E^{\lambda, \psi}_n(\bu^n)\|^2_{\mathcal{H}} + \frac{\epsilon_n}{\eta_n}\|\bu^n - \bu\|_{\mathcal{H}}+ \frac{\epsilon_n^2}{\eta_n}.
\end{align}
\end{lemma}
\begin{proof}
Firstly, notice that 
    \begin{align}
    \begin{split}
    \|\bu^{n+1} - \bu\|^2_{\mathcal{H}} &= \|\bu^{n} - \eta_n\tilde{\nabla}_{\bu}E^{\lambda, \psi}_{n}(\bu^n) - \bu\|^2_{\mathcal{H}}\\
    &= \langle\bu^{n} - \eta_n\tilde{\nabla}_{\bu}E^{\lambda, \psi}_{n}(\bu^n) - \bu, \bu^{n} - \eta_n\tilde{\nabla}_{\bu}E^{\lambda, \psi}_{n}(\bu^n) - \bu\rangle\\
    & =\|\bu^n - \bu\|^2_{\mathcal{H}} - 2\eta_n \langle\bu^n - \bu, \nabla_{\bu}E^{\lambda, \psi}_{n}(\bu^n)\rangle - 2\eta_n\langle\bu^n - \bu, \tilde{\nabla}_{\bu}E^{\lambda, \psi}_{n}(\bu^n) \\
    & \quad - \nabla_{\bu}E^\lambda_{n}(\bu^n)\rangle + \eta_n^2\|\tilde{\nabla}_{\bu}E^{\lambda, \psi}_{n}(\bu^n)\|^2_{\mathcal{H}}
    \end{split}
    \label{eq1_lemma}
    \end{align}
By the Cauchy Schwartz inequality and Lemma \ref{lemma4}, we have 
\begin{align}
| \langle\bu^n - \bu, \tilde{\nabla}_{\bu}E^{\lambda, \psi}_{n}(\bu^n) - \nabla_{\bu}E^\lambda_{n}(\bu^n)\rangle |& \leq \|\bu^n - \bu\|_{\mathcal{H}}\|\tilde{\nabla}_{\bu}E^{\lambda, \psi}_{n}(\bu^n) - \nabla_{\bu}E^{\lambda, \psi}_{n}(\bu^n)\|_{\mathcal{H}}\\
&\leq \frac{\epsilon_n}{\eta_n} \|\bu^n - \bu\|_{\mathcal{H}}.
\label{cauchy}
\end{align}
Substituting Eq. \eqref{cauchy} in Equation \eqref{eq1_lemma} and rearranging terms we obtain
\begin{align}
\label{eq2_lemma}
    \langle \bu^n \!-\! \bu, \nabla_{\bu} E^{\lambda, \psi}_n(\bu^n) \rangle \!\leq\! \frac{\|\bu^n \!-\! \bu\|^2_{\mathcal{H}} -\|\bu^{n+1} \!-\! \bu\|^2_{\mathcal{H}}}{2\eta_n} \!+ \!\frac{\epsilon_n}{\eta_n}\|\bu^n \!-\! \bu\|_{\mathcal{H}} \!+\! \frac{\eta_n}{2}\|\tilde{\nabla}_{\bu}E^{\lambda, \psi}_{n}(\bu^n)\|^2_{\mathcal{H}}.
\end{align}
Then, we have
\begin{align}
E^{\lambda, \psi}_n(\bu^n) - E^{\lambda, \psi}_n(\bu) & \leq  \langle \bu^n - \bu, \nabla_{\bu} E^{\lambda, \psi}_n(\bu^n) \rangle \quad \quad \textup{(By convexity of }E^{\lambda, \psi}_n(\bu^n)), \nonumber \\
& \leq  \!\!\frac{\|\bu^n\! -\! \bu\|^2_{\mathcal{H}} \!-\!\|\bu^{n+1} \!-\! \bu\|^2_{\mathcal{H}}}{2\eta_n} + \frac{\epsilon_n}{\eta_n}\|\bu^n - \bu\|_{\mathcal{H}} + \frac{\eta_n}{2}\|\tilde{\nabla}_{\bu}E^{\lambda, \psi}_{n}(\bu^n)\|^2_{\mathcal{H}}\label{last_eq4}
 \end{align}
Furthermore, 
\begin{align}
\begin{split}
\label{last_eq}
\|\tilde{\nabla}_{\bu}E^{\lambda, \psi}_{n}(\bu^n)\|^2_{\mathcal{H}} &\leq  2\|\tilde{\nabla}_{\bu}E^{\lambda, \psi}_{n}(\bu^n) - {\nabla}_{\bu}E^{\lambda, \psi}_{n}(\bu^n)\|^2_{\mathcal{H}} + 2\|{\nabla}_{\bu}E^{\lambda, \psi}_{n}(\bu^n)\|^2_{\mathcal{H}}\\
 &\leq \frac{2\epsilon_n^2}{\eta_n^2} + 2\|{\nabla}_{\bu}E^{\lambda, \psi}_{n}(\bu^n)\|^2_{\mathcal{H}} \quad \textup{(By Lemma \ref{lemma4})}.
 \end{split}
\end{align}
The Lemma follows from applying Eq \eqref{last_eq} to Eq \eqref{last_eq4}. \hfill
\end{proof}

\begin{lemma}
\label{rate_conv}
Under Assumptions 1-3, for $\epsilon_n = P_2 \eta_n^2$ we have
\[
\Exp\left[E^{\lambda, \psi}(\bu^{N^*}) \!- E^{\lambda, \psi}(\bu^{\lambda, \psi})\right] \leq\! \frac{\|\bu^{\lambda, \psi} \!-\! \bu^0\|_{\mathcal{H}}}{2\sum_{n = 1}^N \eta_n} +\frac{\sum_{n = 1}^N \eta_n^2}{\sum_{n = 1}^N \eta_n\!}\left(\frac{2P_2\kappa C}{\lambda}+4\kappa^2C^2\right)\! + \frac{\sum_{n = 1}^N \epsilon_n^2}{\sum_{n = 1}^N \eta_n}.
\]
\end{lemma}
\begin{proof}
Taking expectation over data and sampling on both sides of Equation \eqref{eq_lemma5} we have 
\begin{align*}
    &\eta_n\Exp\left[E^{\lambda, \psi}(\bu^n)\! - E^{\lambda, \psi}(\bu)\right] \\
    \leq&   \frac{\Exp\left[\|\bu^n \!\!-\! \bu\|^2_{\mathcal{H}}\! -\! \|\bu^{n+1} \! \!- \bu\|^2_{\mathcal{H}}\! + 2\epsilon_n\|\bu^n\! - \!\bu\|_{\mathcal{H}}\right]}{2}\! +\! \mathbb{E}[\eta^2_n\|\nabla_{\bu}E^{\lambda, \psi}_n(\bu^n)\|^2_{\mathcal{H}}]+ {\epsilon_n^2},\\
    \leq& \frac{\Exp\left[\|\bu^n \!- \bu\|^2_{\mathcal{H}}\! - \|\bu^{n+1} \! \!- \bu\|^2_{\mathcal{H}}\! + 2\epsilon_n\|\bu^n\! - \bu\|_{\mathcal{H}}\right]}{2} + 4\eta_n^2\kappa^2C^2 + \epsilon_n^2,
\end{align*}
where the inequality follows form Lemma \ref{lemma_grad_bound}. Summing over $n$ and evaluating at $\bu = \bu^{\lambda, \psi}$ we obtain
\begin{align}
\label{sum_bound}
\begin{split}
    &\sum_{n=0}^{N}\eta_n\Exp\left[E^{\lambda, \psi}(\bu^n) - E^{\lambda, \psi}(\bu^{\lambda, \psi})\right]\\
     \leq& \frac{1}{2}\|\bu^{\lambda, \psi} - \bu^0\|_{\mathcal{H}} + \sum_{n = 0}^N {\epsilon_n}\|\bu^n - \bu^{\lambda, \psi}\|_\mathcal{H} + 4\kappa^2C^2\sum_{n = 0}^N \eta^2_n+ \sum_{n = 0}^N \epsilon^2_n,\\ \leq&  \frac{1}{2}\|\bu^{\lambda, \psi} - \bu^0\|_{\mathcal{H}} + \sum_{n = 0}^N 2\epsilon_n\frac{\kappa C}{\lambda} + 4\kappa^2C^2\sum_{n = 0}^N \eta^2_n+ \sum_{n = 0}^N \epsilon^2_n  \text{ (Lemma \ref{lemma2a})},\\
     =& \frac{1}{2}\|\bu^{\lambda, \psi} - \bu^0\|_{\mathcal{H}} +\left(\frac{2P_1 \kappa C}{\lambda}+4\kappa^2C^2\right)\sum_{n = 0}^N \eta_n^2 + \sum_{n = 0}^N \epsilon_n^2.
\end{split}
\end{align}
By definition of $\bu^*_{\mathcal{S}}$ we then have
 \begin{align}
 \label{min_bound}
 \left(\sum_{n=1}^N\eta_n\right)\Exp\left[E^{\lambda, \psi}(\bu^*_{\mathcal{S}}) - E^{\lambda, \psi}(\bu^{\lambda, \psi})\right] \leq \sum_{n=1}^N\eta_n\Exp\left[E^{\lambda, \psi}(\bu^n) - E^{\lambda, \psi}(\bu^{\lambda,\psi})\right].
 \end{align}
 Dividing by $\sum_{n = 1}^N \eta_n$ on both sides of Eq. \eqref{min_bound} and applying the inequality \eqref{sum_bound} we obtain
\[
\Exp\left[E^{\lambda, \psi}(\bu^{N^*}) \!- E^{\lambda, \psi}(\bu^{\lambda, \psi})\right] \leq\! \frac{\|\bu^{\lambda, \psi} \!-\! \bu^0\|_{\mathcal{H}}}{2\sum_{n = 1}^N \eta_n} +\frac{\sum_{n = 1}^N \eta_n^2}{\sum_{n = 1}^N \eta_n\!}\left(\frac{2P_2\kappa C}{\lambda}+4\kappa^2C^2\right)\! + \frac{\sum_{n = 1}^N \epsilon_n^2}{\sum_{n = 1}^N \eta_n}.
\]
as desired. \hfill
\end{proof}

\begin{lemma}\label{const-violation}
Suppose that there exists a feasible decision $\hat{\bu}$ and a finite constant $C_0$ such that $c(\hat{\bu}(\bz), \bz)\leq C_0$ for all $\bz\in \mathcal{Z}$. Then, 
\begin{align*}
    \lim_{\psi_\rightarrow \infty} \mathbb{E}\left[ \sum_{q = 1}^Q \max \Big(0, g_q\big(\bu^{\lambda, \psi}(\bz)\big)\Big)^2\right] = 0. 
\end{align*}
\end{lemma}
\begin{proof}
By definition of $\bu^{\lambda, \psi}$, we know
\[
\psi \mathbb{E}\left[ \sum_{q = 1}^Q \max \Big(0, g_q\big(\bu^{\lambda, \psi}(\bz)\big)\Big)^2\right]  \leq \mathbb{E}_{\bz}\left[c(\hat{\bu}(\bz), \bz)\right] + \frac{\lambda}{2} \|\hat{\bu}\|^2_{\mathcal{H}}\leq C_0+ \frac{\lambda}{2} \|\hat{\bu}\|^2_{\mathcal{H}}.
\]
Therefore, for any violation tolerance $\delta>0$ we can choose $\psi \geq \frac{2C_0+ \lambda \|\hat{\bu}\|^2_\mathcal{H}}{2\delta}$ to ensure $\mathbb{E}\left[ \sum_{q = 1}^Q \max \Big(0, g_q\big(\bu^{\lambda, \psi}(\bz)\big)\Big)^2\right] \leq \delta$  and the lemma follows.
\end{proof}
\hfill \\

\section{Main Theorems}\label{appendix:theorems}
In this section we state and proof a more general version of Theorem \ref{convergence-rates} and Corollary \ref{corollary}, which correspond to the case $\psi = 0$.
\begin{theorem}[Generalization of Theorem 1]\label{convergence-rates-final}
Let $\bu_{\mathcal{S}}^* \coloneqq \argmin_{\bu \in \{\bu^1, \hdots, \bu^N\}} E_{\mathcal{S}}^{\lambda, \psi}(\bu)$ be the decisions generated by Algorithm \ref{SMOK} when given the set $\mathcal{S} = \{\bz^n\}_{n = 1}^N$ as input, and let $\bu^{\lambda, \psi}$ be the true minimizer of $E^{\lambda,\psi}(\bu)$ over $\mathcal{H}$. If we use constant step-size $\eta$ %with $\eta = \frac{P_1}{\sqrt{N}}<\frac{1}{\lambda}$, and $P_1 >0$,
and constant error bounds $\epsilon = P_2\eta^2$ for some constant $P_2 >0$, then under Assumptions 1-3,  we have that

\[
\Exp\left[E^{\lambda, \psi}(\bu_{\mathcal{S}}^*)  - E^{\lambda, \psi}(\bu^{\lambda, \psi})\right] \leq  \mathcal{O}\left(\frac{\eta}{\lambda}\right).
\]
 %DB   Having the forall quantifier here makes no sense. N is teh data size, no? %KIM fixed this, N is indeed fixed. 
\end{theorem}
\begin{proof}
Applying Lemma \ref{rate_conv} with $\eta_n = \eta$ yields
\begin{align*}
\Exp\left[E^{\lambda, \psi}(\bu_{\mathcal{S}}^*) - E^\lambda(\bu^{\lambda, \psi})\right] \leq \mathcal{O}\left(\frac{\sum_{n = 1}^N \eta_n^2}{\lambda \sum_{n = 1}^N \eta_n}\right)  = \mathcal{O}\left(\frac{\eta}{\lambda }\right),
\end{align*}
as desired. 
\hfill
\end{proof}

\begin{corollary}[Generalization of Corollary 1]\label{corollary-final}
Suppose that there exists a feasible decision $\hat{\bu}$ and finite constants $c_0, C_0$ such that
$c(\hat{\bu}(\bz), \bz)\leq C_0$  and $c_0 \leq c(\bu , \bz)$ for all $\bz\in \mathcal{Z}$ and for all scalar arguments $\bu$. Let $\bu^*$ be the true minimizer of $E(\cdot)$ over $\mathcal{F}$ and let $\bu^\psi$ be the true minimizer of $E^\psi(\cdot)$ over $\mathcal{F}$. If we use constant step-size with $\eta = \frac{P_1}{\sqrt{N}}<\frac{1}{\lambda}$, and $P_1 >0$, constant error bounds $\epsilon = P_2\eta^2$ for some constant $P_2 >0$, and regularization parameter $\lambda$ such that $\lambda \xrightarrow[N\to\infty]{} 0$ and $\lambda \sqrt{N} \xrightarrow[N\to \infty]{}\infty$, then under Assumptions 1-4  we have that\\ %DB do you mean 1-4. you used 1-4 above %KIM yes here we need 1-4, in the previous theorem it's 1-3 since we do not mention constraints feasibility there.

\begin{itemize}
\item[a)] $\lim_{\psi\rightarrow \infty} \lim_{N \rightarrow \infty}\mathbb{E}\left[ \sum_{q = 1}^Q \max \Big(0, g_q\big(\bu_{\mathcal{S}}^*(\bz)\big)\Big)^2\right] = 0$.
\item[b)] $ \lim_{N \rightarrow \infty}\Exp\left[\big|E^\psi(\bu^*_\mathcal{S})  - E^\psi(\bu^\psi)\big|\right] = 0$ for all $\psi>0$.
\item[c)] $\lim_{\psi\rightarrow \infty} \lim_{N \rightarrow \infty}\Exp\left[E(\bu^*_\mathcal{S})\right] \leq E(\bu^*)$.
\end{itemize}

%Since $L_1$ convergence implies convergence in probability, we also obtain that E^{\psi}(\bu^*_\mathcal{S}) \xrightarrow[\mathbb{P}]{}E^{\psi}(\bu^{\psi})$.
\end{corollary}
\begin{proof}
\textit{Part a)} We have
\begin{align*}
    \lim_{N\rightarrow \infty}  \psi \mathbb{E}\left[ \sum_{q = 1}^Q \max \Big(0, g_q\big(\bu_{\mathcal{S}}^*(\bz)\big)\Big)^2\right] &\leq \lim_{N\rightarrow \infty} \Exp\left[E^{\lambda, \psi}(\bu_{\mathcal{S}}^*) \right]& &\!\!\!\!\!\!\!- c_0,\\  & \leq E^{\lambda, \psi}(\bu^{\lambda, \psi}) - c_0& & \text{(by Theorem \ref{convergence-rates})},\\
    &\leq  E^{\lambda, \psi}(\hat{\bu}) -c_0& & \text{(by optimality of }\bu^{\lambda, \psi}),\\
    &= E^{\lambda}(\hat{\bu}) - c_0& & \text{(by feasibility of $\hat{\bu}$)},
\end{align*}
and therefore 
\[
 0\leq \lim_{\psi \rightarrow \infty} \lim_{N\rightarrow \infty}   \mathbb{E}\left[ \sum_{q = 1}^Q \max \Big(0, g_q\big(\bu_{\mathcal{S}}^*(\bz)\big)\Big)^2\right] \leq \lim_{\psi \rightarrow \infty} \frac{E^{\lambda}(\hat{\bu}) - c_0}{\psi} = 0,
\]
\textit{Part b)}
Let $\bu^\psi_{\mathcal{H}}$ be the true minimizer of $E^{\psi}$ over $\mathcal{H}$. Adding and subtracting terms we obtain
\begin{align*}
E^{\psi}(\bu^*_\mathcal{S})  - E^{\psi}(\bu^{\psi}) = &\left(E^{\psi}(\bu^*_\mathcal{S}) - E^{\lambda, \psi}(\bu^*_\mathcal{S})\right) + \left(E^{\lambda, \psi}(\bu_{\mathcal{S}}^*) - E^{\lambda, \psi}(\bu^{\lambda, \psi})\right)\\
 & + \left(E^{\lambda, \psi}(\bu^{\lambda, \psi}) - E^{\psi}(\bu^{\psi}_{\mathcal{H}})\right)+  \left( E^{\psi}(\bu^{\psi}_{\mathcal{H}}) -E^{\psi}(\bu^{\psi})\right).
\end{align*}
The first term on the right hand side is negative, the second term vanishes because of Theorem \ref{convergence-rates}, the third term vanishes with $\lambda$, and the fourth term is zero because we use universal kernels (Assumption \ref{assumption_kernel}). Since $E^{\psi}(\bu^*_\mathcal{S})  - E^{\psi}(\bu^{\psi})$ is non-negative, we obtain
\[\lim_{N\rightarrow \infty}\Exp\left[\big|E^{\psi}(\bu^*_\mathcal{S})  - E^{\psi}(\bu^{\psi})\big|\right]= 0 \quad \forall \psi\geq 0.\]
\textit{Part c)} We have
\begin{align*}
    &\lim_{N\rightarrow \infty} \Exp\left[ E^\psi(\bu_{\mathcal{S}}^*)\right] = \lim_{N\rightarrow \infty} \Exp\left[ E^\psi(\bu_{\mathcal{S}}^*) - E^\psi(\bu^\psi)\right] + E^\psi(\bu^\psi)\\
    \implies &\lim_{\psi \rightarrow \infty} \lim_{N\rightarrow \infty} \Exp\left[ E^\psi(\bu_{\mathcal{S}}^*)\right] = \lim_{\psi \rightarrow \infty}\lim_{N\rightarrow \infty} \Exp\left[ E^\psi(\bu_{\mathcal{S}}^*) - E^\psi(\bu^\psi)\right] +\lim_{\psi \rightarrow \infty} E^\psi(\bu^\psi)\\
    \implies &\lim_{\psi \rightarrow \infty} \lim_{N\rightarrow \infty} \!\Exp\left[ E(\bu_{\mathcal{S}}^*)\right]  = \!\!\lim_{\psi \rightarrow \infty}\lim_{N\rightarrow \infty} \Exp\left[ E^\psi(\bu_{\mathcal{S}}^*) \!- \!E^\psi(\bu^\psi)\right]\!\! +\!\!\!\lim_{\psi \rightarrow \infty} \!\!E^\psi(\bu^\psi)& &\hspace{-1.1cm}\text{(By part a))}\\
    \implies&\lim_{\psi \rightarrow \infty} \lim_{N\rightarrow \infty} \Exp\left[ E(\bu_{\mathcal{S}}^*)\right] \leq  \lim_{\psi \rightarrow \infty}E^\psi(\bu^\psi)& &\hspace{-1.1cm}\text{(By part b))}\\
    \implies&\lim_{\psi \rightarrow \infty} \lim_{N\rightarrow \infty} \Exp\left[ E(\bu_{\mathcal{S}}^*)\right] \leq  E(\bu^*),& &\hspace{-2.4cm}\text{(By optimality of $\bu^\psi$)}\\
\end{align*}
as desired.
\hfill \end{proof}
\section{Finding Lower Bounds}
\label{lower_bound}
We emphasize that we only need to find lower bounds for the case in which the dimension of the data and the dimension of the controls is equal to $1$; since the experiments run for multidimensional cases were designed to have the same objective value as the one dimensional case. The exact problem we want to lower bound is then
\begin{align}
    \min_{\bu_{1:T}} \quad & \mathbb{E}_{ \bw | \bx}\left[\sum_{t = 1}^T 2\left[s_t\right]^+ + \left[-s_t\right]^+ \bigg \vert \: \bx = \bx_0   \right]\\
    \text{s.t.}\quad  & s_t = s_{t-1} + u_t - w_t\\
    & u_t \geq 0& &\forall t\in [T],\label{lowerbound_constraint}\\
    & u_t \leq 150& &\forall t\in [T], \label{upperbound_constraint}\\
    & u_t + u_{t+1} \leq 200& &\forall t\in [T-1].
\end{align}
The demands $w_t$ were generated as a linear function of the covariates with some added noise; specifically, $w_t = \alpha_t x + \epsilon_t$, where $\epsilon_t$ was sampled from a standard distribution and the constants $\alpha_t$ were selected to be close to $50$ in order for the triangular constraints to be relevant. In fact, for the specified parameters, we found that the constraints \eqref{upperbound_constraint} and \eqref{lowerbound_constraint} are quite lose and we can find a good lower bound for the the optimal objective value by removing these constraints. The problem to solve can then be simplified as 
\begin{align}
    \min_{\bu_{1:T}}  & \mathbb{E}_{\boldsymbol{\epsilon}}\left[\sum_{t = 1}^T 2\left[\sum_{i = 1}^t u_i(x_0, \boldsymbol{\epsilon}_{1:i-1}) - \alpha_ix_0 - \epsilon_i\right]^+ \!\!\!+  \left[-\sum_{i = 1}^t u_i(x_0, \boldsymbol{\epsilon}_{1:i-1})  - \alpha_i x_0 - \epsilon_i \right]^+  \right]\nonumber \\
    \text{s.t.}\quad  & u_t(x_0, \boldsymbol{\epsilon}_{1:t-1}) + u_{t+1}(x_0, \boldsymbol{\epsilon}_{1:t}) \leq 200 \quad \forall t\in [T-1]. \label{lower_bound_problem}
\end{align}
To solve the problem above, we use the fact that if $\epsilon$ has a normal distribution then the function $f(a) = \mathbb{E}_\epsilon\left[2[a - \epsilon]^+ + [\epsilon - a]^+\right]$ is strictly convex and 
\begin{align}
   a_0 \coloneqq &\argmin_a \: \mathbb{E}_\epsilon\left[2[a - \epsilon]^+ + [\epsilon - a]^+\right] \\
   =& \min_a \: \int_{-\infty}^a 2(a - x)\frac{e^{-x^2/2}}{\sqrt{2\pi}}\,dx  + \int_a^\infty (x - a)\frac{e^{-x^2/2}}{\sqrt{2\pi}}\,dx \nonumber\\
   =& \min_a \: \frac{3e^{-a^2/2}}{\sqrt{2\pi}} + \frac{3}{2}a\left(1 + \text{erf}\left({\frac{a}{\sqrt{2}}}\right)\right) - a \nonumber \\
   =& -\sqrt{2}\text{erf}^{-1}\left(\frac{1}{3}\right).
   \label{optimum_a}
\end{align}
We will show how to exactly solve problem \eqref{lower_bound_problem} in the case $T=2$ (the analysis is similar for cases $T=3, 4, 5$). Suppose then that $T=2$. For a fix value of $u_1(x_0)$, the optimization problem \eqref{lower_bound_problem} over $u_2(x_0, \epsilon_1)$ becomes 
\begin{align*}
    \min_{u_2}&  \mathbb{E}_{\epsilon_2}\left[ 2\left[\sum_{i = 1}^2 u_i(x_0, \boldsymbol{\epsilon}_{1:i-1}) - \alpha_ix_0 - \epsilon_i\right]^+ \!\!+  \left[-\sum_{i = 1}^2 u_i(x_0, \boldsymbol{\epsilon}_{1:i-1})  - \alpha_i x_0 - \epsilon_i \right]^+  \right]\\
    \text{s.t.}\quad  & u_2(x_0, \boldsymbol{\epsilon}_{1})\leq 200 - u_1(x_0).
\end{align*}
Applying the result from Eq. \eqref{optimum_a} we obtain 
\[u^*_2(x_0, \epsilon_1) = \min\{(\alpha_1 + \alpha_2)x_0 + \epsilon_1 - u_1(x_0)+ a_0, 200 - u_1(x_0)\},\]
which implies that the term $\sum_{i = 1}^2 u_i(x_0, \boldsymbol{\epsilon}_{1:i-1}) - \alpha_ix_0 - \epsilon_i$ evaluated at $u^*_2(x_0, \epsilon_1)$ is equal to $\min\{ a_0 - \epsilon_2, 200 -(\alpha_1 + \alpha_2)x_0 - \epsilon_1 -\epsilon_2)$, which is  independent of $u_1(x_0)$. We can then find the optimal $u_1$ by solving
\begin{align*}
    \min_{u_1} \quad & \mathbb{E}_{\epsilon_1}\left[2\left[ u_1(x_0) - \alpha_1x_0 - \epsilon_1\right]^+ +  \left[\alpha_1 x_0 + \epsilon_1 - u_1(x_0) \right]^+  \right]\\
    \text{s.t.}\quad  & u_1(x_0,)\leq 200 ,
\end{align*}
which again, using Eq. \eqref{optimum_a} yields $u^*_1(x_0) = \min\{\alpha_1 x_0 + a_0, 200\}$.

% Appendix here
% Options are (1) APPENDIX (with or without general title) or
%             (2) APPENDICES (if it has more than one unrelated sections)
% Outcomment the appropriate case if necessary
%
% \begin{APPENDIX}{<Title of the Appendix>}
% \end{APPENDIX}
%
%   or
%
% \begin{APPENDICES}
% \section{<Title of Section A>}
% \section{<Title of Section B>}
% etc
% \end{APPENDICES}

\end{appendix}

\end{document}